\documentclass[reqno, a4paper, 10pt]{amsart}
\usepackage{amssymb,url, color, mathrsfs, multirow, makecell, array, caption}
\usepackage[colorlinks=true, bookmarks=true, pdfstartview=FitH, pagebackref=true, linktocpage=true]{hyperref}
\usepackage[short,nodayofweek]{datetime}
\usepackage[pagewise,running,mathlines,displaymath, switch]{lineno}
\usepackage{times}
\usepackage{indentfirst, tabularx}
\usepackage[table]{xcolor}
\usepackage{float, hhline, colortbl}
\usepackage{graphicx}
\usepackage{longtable}
\usepackage{float}
\usepackage{caption}
\usepackage{enumitem}

\newcolumntype{C}[1]{>{\centering\arraybackslash }b{#1}}

\hypersetup{
pdftitle={Exhaustive existence and non-existence results for Hardy-Henon equations in R^n},
pdfauthor={Yoshikazu Giga and Quoc Anh Ngo},
colorlinks = true,
linkcolor = magenta,
citecolor = blue,
}

\theoremstyle{plain}
\numberwithin{equation}{section}
\newtheorem{theorem}{Theorem}[section]
\newtheorem{lemma}{Lemma}[section]
\newtheorem{corollary}{Corollary}[section]

\theoremstyle{remark}

\newtheorem{remark}{Remark}[section]

\DeclareMathOperator{\R }{\mathbf R}
\newcommand{\psnma}{p_{\mathsf S}(\sigma)}
\newcommand{\psnmz}{p_{\mathsf S}(0)}
\newcommand{\ps}{p_{\mathsf S}}

\definecolor{brown}{rgb}{0.5,0,0}
\definecolor{backgroundcolor}{rgb}{0.98, 0.92, 0.73}

\def\cfac#1{\ifmmode\setbox7\hbox{$\accent"5E#1$}\else\setbox7\hbox{\accent"5E#1}\penalty 10000\relax\fi\raise 1\ht7\hbox{\lower1.05ex\hbox to 1\wd7{\hss\accent"13\hss}}\penalty 10000\hskip-1\wd7\penalty 10000\box7 }

\author[Y. Giga]{Yoshikazu Giga}
\address[Y. Giga]{
Graduate School of Mathematical Sciences, The University of Tokyo, 3-8-1 Komaba, Meguro-ku, Tokyo 153-8914, Japan}
\email{\href{mailto: Y. Giga <labgiga@ms.u-tokyo.ac.jp>}{labgiga@ms.u-tokyo.ac.jp}}

\author[Q.A. Ng\^o]{Qu\cfac oc Anh Ng\^o}
\address[Q.A. Ng\^o]{
University of Science, Vietnam National University, Hanoi, Vietnam}
\email{\href{mailto: Q.A. Ng\^o <nqanh@vnu.edu.vn>}{nqanh@vnu.edu.vn}}

\definecolor{LightCyan}{rgb}{0.88,1,1}

\makeatletter
\def\@cite#1#2{[\textbf{#1}\if@tempswa, #2\fi]}
\makeatother

\begin{document}

\allowdisplaybreaks

\title[Exhaustive existence and non-existence results for Hardy--H\'enon equations in $\R ^n$]
{Exhaustive existence and non-existence results for Hardy--H\'enon equations in $\R ^n$}

\begin{abstract}
This paper concerns solutions to the Hardy--H\'enon equation
\[
-\Delta u = |x|^\sigma u^p 
\]
in $\R ^n$ with $n \geq 1$ and arbitrary $p, \sigma \in \R $. This equation was proposed by H\'enon in 1973 as a model to study rotating stellar systems in astrophysics. Although there have been many works devoting to the study of the above equation, at least one of the following three assumptions $p>1$, $\sigma \geq -2$, and $n \geq 3$ is often assumed. The aim of this paper is to investigate the equation in other cases of these parameters, leading to a complete picture of the existence/non-existence results for non-trivial, non-negative solutions in the full generality of the parameters. In addition to the existence/non-existence results, the uniqueness of solutions is also discussed.
\end{abstract}

\date{\bf \today \; at \, \currenttime}



\dedicatory{\it Dedicated to Professor Eiji Yanagida on the occasion of his 65th birthday}

\maketitle

\section{Introduction}

Over the last several decades, elliptic equations of the following form
\begin{equation}\label{eqn}
-\Delta u = K(x) u^p
\end{equation}
have captured much attention because they arise in various branches of pure and applied mathematics. For example, the simple model when $K \equiv 1$, also known as the Lane--Emden equation, has its root in astrophysics because the function $u$ can be thought of the density of stars. In this note, we are interested in non-trivial, non-negative, classical solutions to the following special case of \eqref{eqn}
\begin{subequations}\label{eqMAIN}
\begin{align}
-\Delta u = |x|^\sigma u^p \quad \text{in } \R ^n,
\tag*{ \eqref{eqMAIN}$_\sigma$}
\end{align}
\end{subequations}
where $n \geq 1$ and $p \in \R $ is a parameter. Equation \eqref{eqMAIN}$_\sigma$ was first proposed by H\'enon in 1973 as a model to study rotating stellar systems. Traditionally, the equation \eqref{eqMAIN}$_\sigma$ is called the H\'enon (resp. Hardy or Lane--Emden) equation if $\sigma > 0$ (resp. $\sigma <0$ or $\sigma =0$). However, as we are mostly interested in $\sigma \ne 0$ without any sign restriction we shall call \eqref{eqMAIN}$_\sigma$ the Hardy--H\'enon equation. 

So far, apart from the non-negativity of solutions, we have not precisely mentioned the definition of solutions that we are interested in, especially, we have not yet mentioned what `classical solution' really means. Usually, by a classical solution $u$ to \eqref{eqMAIN}$_\sigma$ we mean $u \in C^2(\R^n)$. However, the presence of $\sigma$ and $\alpha$ plays some role forcing us to work with solutions in suitable sense. For example, in the case $\sigma <0$ it does not make sense for \eqref{eqMAIN}$_\sigma$ to define at the origin. Therefore, let us mention precisely the notion of solution involved here. Throughout this paper, a function $u$ is called \textbf{classical solution} to \eqref{eqMAIN}$_\sigma$ if it belongs to
\begin{equation}\label{eqClassicalSolutionClass}
C^2 (\R^n) \;\; \text{ if } \sigma \geq 0, \quad
C(\R^n) \cap C^2 (\R^n \backslash \{ 0 \} ) \;\; \text{ if } \sigma < 0;
\end{equation}
and the equation \eqref{eqMAIN}$_\sigma$ is fulfilled pointwisely except at $x=0$ for $\sigma < 0$. Having this definition, in the case $\sigma \geq 0$, the above notion of classical solutions coincides with the standard one. For simplicity, from now on, by solution we mean classical solution.

Let us now briefly go through some literature review for solutions to the equation \eqref{eqMAIN}$_\sigma$. Equipped with \eqref{eqMAIN}$_\sigma$ is the important number $\psnma$, called the critical Sobolev exponent, which is given by
\[
\psnma =
\begin{cases}
\dfrac{n+2 + 2\sigma}{n-2 } & \text{ if } n > 2,\\
+\infty & \text{ if } n \leq 2.
\end{cases}
\]
This critical exponent, generalizing the classical ones for the case $\sigma = 0$, is important because the solvability of the equation \eqref{eqMAIN}$_\sigma$ often changes when $p$ passes through it. 

For the autonomous case, that is $\sigma = 0$, the existence of solutions to the equation \eqref{eqMAIN}$_0$ is now well-understood for general $p$, and there have been many works devoting to \eqref{eqMAIN}$_0$. For example, it is well-known that \eqref{eqMAIN}$_0$ has no classical solution if $1 < p < \psnmz$, see \cite{GS81, CLi91}. On the other hand, if $p \geq \psnmz,$ the equation \eqref{eqMAIN}$_0$ always has positive classical solutions; see for instance \cite{GS81, LGZ06b}. In the remaining cases, namely $p\leq 1$ and $n \geq 1$ we still have non-existence results. For example, if $n \leq 2$, we simply apply the maximum principle to super-harmonic functions. In the case $n \geq 3$ and $p <0$, one can use a trick via tackling the auxiliary function $u(x) + \varepsilon |x|^2$ and let $\varepsilon \searrow 0$. In the case $n\geq 3$ and $0<p<1$, this can be achieved by making use of the method developed in \cite{BK92} and in the case $p=1$ this is done by some simple integral estimate. It is worth noting that existence and non-existence results for higher order equations in $\R ^n$ have already been studied. These equations are slightly different from the classical equation \eqref{eqMAIN}$_0$ by replacing $-\Delta$ by $\Delta^m$ for any $m \in \{1,2,3,..\}$. For interested readers, we refer to the recent work \cite{NNPY18} for a complete picture of existence and non-existence results.

For the non-autonomous case, namely $\sigma \ne 0$, although there have also been many works devoting to the equation, the situation is not well-understood, especially there are some cases for which no result is known, and this is the starting point of this work. Among others, Ni proved in \cite{Ni82, Ni86} the existence of classical solution for 
\[
p \geq \ps ( \sigma) \quad \text{and} \quad \sigma > -2.
\] 
Hence, we are left with the subcritical case $p < \ps (\sigma)$. As far as we know, the subcritical case $1<p<\ps ( \sigma)$ was firstly classified by Reichel and Zou. In \cite{RZ00}, they considered a cooperative semilinear elliptic system with refinement of the moving sphere method. The result of Reichel and Zou indicates that \eqref{eqMAIN}$_\sigma$ does not admit any classical solution if 
\[
1<p<\ps( \sigma ) \quad \text{and} \quad \sigma > -2;
\] 
see \cite[Theorem 2]{RZ00}. On the other hand, following the results due to Mitidieri and Pohozaev in \cite[Theorem 6.1]{MP01}, Dancer, Du and Guo in \cite[Theorem 2.3]{DDG11}, Brezis and Cabr\'e in \cite{BC98}, the condition $\sigma > -2$ is necessary for the existence of solutions to \eqref{eqMAIN}$_\sigma$. Thus, we have a complete picture for the existence problem of solutions to \eqref{eqMAIN}$_\sigma$ in the case $p > 1$, namely, we must assume both 
$$\sigma > -2 \quad \text{and} \quad p \geq \ps ( \sigma).$$ 
In the case of higher order Hardy--H\'enon equations, we refer the reader to \cite{NY20}. 

From the above discussion, we are left with the case $n \geq 1$, $p \leq 1$, and $\sigma \in \R$. This leads us to a recent work due to Dai and Qin, which also motivates us to write this note; see \cite{DQ20}. To be more precise, the authors in \cite{DQ20} consider solutions to \eqref{eqMAIN}$_\sigma$ under the two assumptions $0 < p \leq 1$ and $n \geq 3$ and obtain a non-existence result for solutions to \eqref{eqMAIN}$_\sigma$ for all $\sigma \in \R$. To the best of our knowledge, there is no result for the case $p \leq 0$ and general $\sigma\in R$. In fact, putting the above results together, we are left with the following two cases: 
\begin{equation}\label{caseA}
\left.
\begin{aligned}
& p \leq 0, & \sigma \in \R, & \quad \text{ and } \quad n \geq1, \qquad\\
& 0 < p \leq 1, & \sigma \in \R, & \quad \text{ and } \quad n=1,2. 
\end{aligned}
\right\}
\tag{I}
\end{equation} 
In this work, by mainly considering the above cases \eqref{caseA}, we aim to complete the picture of existence/non-existence results for non-trivial, non-negative, classical solutions to \eqref{eqMAIN}$_\sigma$. As the cases $n \geq 2$ and $n=1$ are quite different, our main result consists of two theorems, one is for $n \geq 2$ that we are going to describe soon, and the other is for $n=1$.

For the case $n \geq 2$, our result is as follows.

\begin{theorem}\label{thm-MAIN}
Let $n \geq 2$, $p \in \R$, and $\sigma \in \R$. Then, the equation \eqref{eqMAIN}$_\sigma$ admits at least one non-trivial, non-negative solution in the sense of \eqref{eqClassicalSolutionClass} if, and only if, 
$$n \geq 3, \quad p \geq \frac{n+2}{n-2}, \quad \text{and} \quad \sigma >-2.$$
\end{theorem}

For the reader convenience, we summary the result described in Theorem \ref{thm-MAIN} in Table \ref{table-1} below.

{\scriptsize
\begin{center}
\begin{longtable}{
>{\centering\arraybackslash}p{.055\textwidth}||
>{\centering\arraybackslash}p{.13\textwidth}|
>{\centering\arraybackslash}p{.13\textwidth}|
>{\centering\arraybackslash}p{.13\textwidth}|
>{\centering\arraybackslash}p{.19\textwidth}|
>{\centering\arraybackslash}p{.19\textwidth}
} 
\hline\rowcolor{LightCyan}
& $p < 0 $ & $p = 0$ 	& $0 < p \leq 1$	& \multicolumn{2}{c}{$p>1 $}\\
\hline
$n = 2$ & \textbf{NO} & \textbf{NO} & \textbf{NO} & \multicolumn{2}{c}{\textbf{NO} }\\
& Sec. \ref{sec-n>=2,p<0} & Sec. \ref{sec-n>=2,p=0} & sect. \ref{sec-n=2,0<p<=1} & \multicolumn{2}{c}{ $\sigma > -2$: \cite{RZ00}, $\sigma \leq -2$: \cite{DDG11}} \\
\hline\rowcolor{LightCyan}
& $p < 0 $ & $p = 0$ 	& $0 < p \leq 1$	& $1< p <\frac{n+2+2\sigma}{n-2}$ 	& $p \geq \frac{n+2+2\sigma}{n-2}$	\\
\hline
$n\geq 3$ & \textbf{NO} & \textbf{NO} 
& \textbf{NO} & \textbf{NO} &\textbf{YES} if $\sigma > -2$ \\
& Sec. \ref{sec-n>=2,p<0} & Sec. \ref{sec-n>=2,p=0} & \cite{DQ20} & $\sigma >-2$: \cite{Ni82, Ni86} & $\sigma >-2$: \cite{Ni82} \\
& & & & $\sigma \leq -2$: \cite{DDG11} & $\sigma \leq -2$: \cite{DDG11}\\
\hline
\caption{Results for $-\Delta u = |x|^\sigma u^p$ in $\R^n$ with $n \geq 2$.}
\label{table-1}
\end{longtable}
\end{center}
}

\bigskip
Now we consider the case $n=1$ and as we shall see later this case is very special. Note that in this case, our equation \eqref{eqMAIN}$_\sigma$ becomes
\begin{equation}\label{eqMAIN-n=1-R}
-u''(x) = |x|^\sigma u(x)^p \quad \text{in } \R,
\end{equation}
and it is understood with $x \in \R$ if $\sigma \geq 0$ and $x \ne 0$ if $\sigma < 0$. Let us assume for a moment that $p \leq 0$. For each $a\in (0,1)$ and $C>0$, direct computation shows 
$$-(C |x|^a)'' = C a(1-a) |x|^{a-2}= C^\frac{2+\sigma}{a} a(1-a) |x|^\sigma \big( C |x|^a \big)^\frac{a-2-\sigma}{a}.$$ 
Hence, the even function 
$$ u_a (x) :=c_a |x|^a$$ 
solves \eqref{eqMAIN}$_\sigma$ in $\R$ with
$$a= (2+\sigma)/(1-p), \quad c_a = \big( a ( 1-a ) \big)^{1/(p-1)}.$$
Note that the condition $a \in (0,1)$ tells us that either $\sigma > -2$ and $p<-1-\sigma$ or $\sigma < -2$ and $p>-1-\sigma$. Hence, it follows from the above calculation that there are solutions to \eqref{eqMAIN}$_\sigma$ in $\R$ in the following regime
\begin{itemize}
 \item either $0> \sigma > -2$ and $p<-1-\sigma$
 \item or $\sigma < -2$ and $p>-1-\sigma$.
\end{itemize}
We note that we need $\sigma<0$ because $a<1$. Hence, we are forced to consider \eqref{eqMAIN}$_\sigma$ in $\R$ in the following opposite regime
\begin{itemize}
 \item $\sigma \geq 0$ and $p \in \R$,
 \item $\sigma > -2$ and $p \geq -1-\sigma$,
 \item $\sigma = -2$ and $p \in \R$,
 \item $\sigma < -2$ and $p\leq-1-\sigma$.
\end{itemize}
We should emphasize that in the above four cases, there are some partly known results. For example, if $\sigma \geq -2$ and $p>1$, then we have already known that the equation has no solution; see \cite{MP01} for the case $\sigma = -2$ and \cite{PS12} for the case $\sigma >-2$. Hence, the remaining cases that we actually have to consider is the following
\begin{equation}\label{caseB}
\left.
\begin{aligned}
 & \sigma \leq -2 & \quad \text{ and } \quad p \leq -1-\sigma, \qquad\\
 & 0>\sigma > -2 & \quad \text{ and } \quad -1-\sigma \leq p \leq 1, \qquad\\
 & \sigma \geq 0 & \quad \text{ and } \quad p \leq 1, \qquad\\
\end{aligned}
\right\}
\tag{II}
\end{equation} 
By giving a complete answer to the cases in \eqref{caseB}, we obtain the following complete picture of existence/non-existence of solutions to \eqref{eqMAIN-n=1-R}.

\begin{theorem}\label{thm-MAIN=n=1-R}
Let $p \in \R$ and $\sigma \in \R$. Then, the equation \eqref{eqMAIN-n=1-R} in $\R$ admits at least one solution in the sense of \eqref{eqClassicalSolutionClass} if, and only if, one of the following two cases occurs
\begin{enumerate}
 \item either $\sigma < -2$ and $p> -1-\sigma$
 \item or $-2 < \sigma <0$ and $p<-1-\sigma$.
\end{enumerate}
\end{theorem}

We summary the above existence/non-existence results described in Theorem \ref{thm-MAIN=n=1-R} in Table \ref{table-n=1-R} below. 

{\scriptsize
\begin{center}
\begin{longtable}{
>{\centering\arraybackslash}p{.14\textwidth}||
>{\centering\arraybackslash}p{.115\textwidth}|
>{\centering\arraybackslash}p{.115\textwidth}|
>{\centering\arraybackslash}p{.115\textwidth}|
>{\centering\arraybackslash}p{.115\textwidth}|
>{\centering\arraybackslash}p{.115\textwidth}|
>{\centering\arraybackslash}p{.115\textwidth}
}
\hline\rowcolor{LightCyan}
&
\multicolumn{2}{c|}{$p \leq 0 $}
	&\multicolumn{2}{c|}{$0<p <1 $}
	& $1 \leq p \leq -1-\sigma$	
	& $p>-1-\sigma$ \\
\hline
$\sigma < -2$ 
& \multicolumn{2}{c|}{\textbf{NO}} 
&\multicolumn{2}{c|}{ \textbf{NO} } 
& \textbf{NO} 
& \textbf{YES} \\
 & \multicolumn{2}{c|}{sect. \ref{sec-s<-2,p<=0}}
 & \multicolumn{2}{c|}{sect. \ref{sec-n=1,s<=-2,0<p<=1}} 
 & sect. \ref{sec-n=1,s<=-2,0<=p<=-1-s} 
 & $C|x|^\frac{2+\sigma}{1-p}$ \\	
\hline\rowcolor{LightCyan}
&
\multicolumn{2}{c|}{$p \leq 0 $}
	&\multicolumn{2}{c|}{$0<p <1 $}
	& \multicolumn{2}{c}{$p \geq 1 $} \\
\hline
$\sigma= -2$ 
& \multicolumn{2}{c|}{\textbf{NO}} 
&\multicolumn{2}{c|}{ \textbf{NO} } 
& \multicolumn{2}{c}{\textbf{NO}} \\
 & \multicolumn{2}{c|}{sect. \ref{sec-s<-2,p<=0}}
 & \multicolumn{2}{c|}{sect. \ref{sec-n=1,s<=-2,0<p<=1}} 
 & \multicolumn{2}{c}{\cite[Theo. 8.1]{MP01}, sect. \ref{sec-n=1,s=-2,p=1}} \\	 
\hline\rowcolor{LightCyan}
&
\multicolumn{2}{c|}{$p \leq 0 $}
	& $0< p < -1-\sigma $ 
	& $-1 - \sigma \leq p < 1$	
	& \multicolumn{2}{c}{$p \geq 1 $}	\\
\hline	
$-2 < \sigma <-1$ 
& \multicolumn{3}{c|}{\textbf{YES}} 
& \textbf{NO} 
& \multicolumn{2}{c}{\textbf{NO}}\\
& \multicolumn{3}{c|}{$C|x|^\frac{2+\sigma}{1-p}$} 
& sect. \ref{sec-n=1,-2<s<-1,-1-s<p<1} 
& \multicolumn{2}{c}{\cite[Prop. A.1]{PS12}}\\
\hline\rowcolor{LightCyan}
	& $p < -1 - \sigma $ 
	& $-1 - \sigma \leq p \leq 0$	
& \multicolumn{2}{c|}{$0< p <1 $}
&\multicolumn{2}{c}{$p \geq 1 $}	\\
\hline	
$-1 \leq \sigma <0$ 
& \textbf{YES} 
& \textbf{NO} 
& \multicolumn{2}{c|}{\textbf{NO}} 
 & \multicolumn{2}{c}{\textbf{NO}}\\
& $C|x|^\frac{2+\sigma}{1-p}$ 
& sect. \ref{sec-s<=-1,-1-s<=p<=0} 
& \multicolumn{2}{c|}{sect. \ref{sec-n=1,-1<=s,0<p<1}} 
& \multicolumn{2}{c}{\cite[Prop. A.1]{PS12}, sect. \ref{sec-n=1,-1<=s,0<p<1}}\\
\hline
\rowcolor{LightCyan}
	& $p < -1 - \sigma $ 
	& $-1 - \sigma \leq p \leq 0$	
& \multicolumn{2}{c|}{$0< p <1 $}
&\multicolumn{2}{c}{$p \geq 1 $}	\\
\hline	
$\sigma \geq 0$ 
& \multicolumn{4}{c|}{\textbf{NO}} 
& \multicolumn{2}{c}{\textbf{NO}} \\
& \multicolumn{4}{c|}{sect. \ref{ssect-thm-n=1-R}}
& \multicolumn{2}{c}{\cite[Prop. A.1]{PS12}, sect. \ref{ssect-thm-n=1-R}}\\
\hline
\caption{Results for $- u'' (x)= |x|^\sigma u(x)^p$ in $\R$.}
\label{table-n=1-R}
\end{longtable}
\end{center}
}

\bigskip
From Table \ref{table-n=1-R} above, it is immedate to notify that our equation \eqref{eqMAIN-n=1-R} does not admit any non-negative $C^2$-solution if $\sigma \geq 0$. Hence, to make our equation more interesting, we only consider \eqref{eqMAIN-n=1-R} on the half line $[0, +\infty)$, namely, we are interested in non-negative, non-trivial solution to
\begin{equation}\label{eqMAIN-n=1}
-u''(x) = |x|^\sigma u(x)^p \quad \text{in } (0,+\infty).
\end{equation}
A function $u$ is called \textbf{classical solution} to \eqref{eqMAIN-n=1} if it belongs to $C^2 (0,+\infty)$. (This is also known as punctured solutions in \cite{NY20}.) Our next result is a complete picture of existence/non-existence of classical solutions to \eqref{eqMAIN-n=1}.

\begin{theorem}\label{thm-MAIN=n=1}
Let $p \in \R$ and $\sigma \in \R$. Then, the equation \eqref{eqMAIN-n=1} in $(0, +\infty)$ admits at least one $C^2$-solution if, and only if, one of the following two cases occurs
\begin{enumerate}
 \item either $\sigma < -2$ and $p> -1-\sigma$
 \item or $\sigma >-2$ and $p<-1-\sigma$.
\end{enumerate}
\end{theorem}

We summary the above existence/non-existence results described in Theorem \ref{thm-MAIN=n=1} in Table \ref{table-n=1} below. 

{\scriptsize
\begin{center}
\begin{longtable}{
>{\centering\arraybackslash}p{.14\textwidth}||
>{\centering\arraybackslash}p{.115\textwidth}|
>{\centering\arraybackslash}p{.115\textwidth}|
>{\centering\arraybackslash}p{.115\textwidth}|
>{\centering\arraybackslash}p{.115\textwidth}|
>{\centering\arraybackslash}p{.115\textwidth}|
>{\centering\arraybackslash}p{.115\textwidth}
}
\hline\rowcolor{LightCyan}
&
\multicolumn{2}{c|}{$p \leq 0 $}
	&\multicolumn{2}{c|}{$0<p <1 $}
	& $1 \leq p \leq -1-\sigma$	
	& $p>-1-\sigma$ \\
\hline
$\sigma < -2$ 
& \multicolumn{2}{c|}{\textbf{NO}} 
&\multicolumn{2}{c|}{ \textbf{NO} } 
& \textbf{NO} 
& \textbf{YES} \\
 & \multicolumn{2}{c|}{sect. \ref{sec-s<-2,p<=0}}
 & \multicolumn{2}{c|}{sect. \ref{sec-n=1,s<=-2,0<p<=1}} 
 & sect. \ref{sec-n=1,s<=-2,0<=p<=-1-s} 
 & $C|x|^\frac{2+\sigma}{1-p}$ \\	
\hline\rowcolor{LightCyan}
&
\multicolumn{2}{c|}{$p \leq 0 $}
	&\multicolumn{2}{c|}{$0<p <1 $}
	& \multicolumn{2}{c}{$p \geq 1 $} \\
\hline
$\sigma= -2$ 
& \multicolumn{2}{c|}{\textbf{NO}} 
&\multicolumn{2}{c|}{ \textbf{NO} } 
& \multicolumn{2}{c}{\textbf{NO}} \\
 & \multicolumn{2}{c|}{sect. \ref{sec-s<-2,p<=0}}
 & \multicolumn{2}{c|}{sect. \ref{sec-n=1,s<=-2,0<p<=1}} 
 & \multicolumn{2}{c}{\cite[Theo. 8.1]{MP01}, sect. \ref{sec-n=1,s=-2,p=1}} \\	 
\hline\rowcolor{LightCyan}
&
\multicolumn{2}{c|}{$p \leq 0 $}
	& $0< p < -1-\sigma $ 
	& $-1 - \sigma \leq p < 1$	
	& \multicolumn{2}{c}{$p \geq 1 $}	\\
\hline	
$-2 < \sigma <-1$ 
& \multicolumn{3}{c|}{\textbf{YES}} 
& \textbf{NO} 
& \multicolumn{2}{c}{\textbf{NO}}\\
& \multicolumn{3}{c|}{$C|x|^\frac{2+\sigma}{1-p}$} 
& sect. \ref{sec-n=1,-2<s<-1,-1-s<p<1} 
& \multicolumn{2}{c}{\cite[Prop. A.1]{PS12}}\\
\hline\rowcolor{LightCyan}
	& $p < -1 - \sigma $ 
	& $-1 - \sigma \leq p \leq 0$	
& \multicolumn{2}{c|}{$0< p <1 $}
&\multicolumn{2}{c}{$p \geq 1 $}	\\
\hline	
$\sigma \geq -1 $ 
& \textbf{YES} 
& \textbf{NO} 
& \multicolumn{2}{c|}{\textbf{NO}} 
 & \multicolumn{2}{c}{\textbf{NO}}\\
& $C|x|^\frac{2+\sigma}{1-p}$ 
& sect. \ref{sec-s<=-1,-1-s<=p<=0} 
& \multicolumn{2}{c|}{sect. \ref{sec-n=1,-1<=s,0<p<1}} 
& \multicolumn{2}{c}{\cite[Prop. A.1]{PS12}, sect. \ref{sec-n=1,-1<=s,0<p<1}}\\
\hline
\caption{Results for $- u'' (x)= |x|^\sigma u(x)^p$ in $(0,+\infty)$.}
\label{table-n=1}
\end{longtable}
\end{center}
}

\bigskip

\begin{remark}\label{rmk-PS12}
From Tables \ref{table-n=1} and \ref{table-n=1-R} above we have some remarks concerning the work \cite{PS12}.

\begin{enumerate}
 \item The equation considered in \cite{PS12} is in $(b,+\infty)$ for any $b \geq 0$ and solution is only assumed to be of $C^2$ in $(b,+\infty)$. Hence, solution in the sense of \cite{PS12} is not necessarily continuous up to $b$. This is very important for us to make use of duality which will be described below.
 
 \item Although the result in \cite{PS12} for the case $n=1$ is stated for $p>1$, it is quite clear that their argument actually works well for the case $p=1$.
\end{enumerate} 
\end{remark}
 
As mentioned in the preceding remark, in the proof of Theorems \ref{thm-MAIN=n=1-R} and \ref{thm-MAIN=n=1}, we make use some duality. To be more precise, if $u$ solves $$-u''(r) = r^\sigma u(r)^p$$ in $(0,+\infty)$, then the Kelvin type transformation $v : r \mapsto ru(1/r)$ solves 
$$
-v''(r) = r^{-p-\sigma - 3} v(v) ^p
$$
in $(0, +\infty)$. Using this, it is shown in section \ref{sec-n=1,-2<s<-1,-1-s<p<1} that the non-existence result in the case $\sigma \in (-2,-1)$ and $p \in (-1-\sigma, 1)$ follows from the non-existence result in the case $\sigma \leq -2$ and $0<p<1$ which is established earlier in section \ref{sec-n=1,s<=-2,0<p<=1}. Another example is that the non-existence result in the case $\sigma < -2$ and $1 \leq p \leq -1-\sigma$ studied in section \ref{sec-n=1,s<=-2,0<=p<=-1-s} can be obtained from the similar result in the case $\sigma \geq -2$ and $p \geq 1$.

Equations of the form \eqref{eqMAIN}$_\sigma$ with $p \ne 1$ and $\sigma \ne -2$ enjoy some invariant properties, namely, if $u$ solves \eqref{eqMAIN}$_\sigma$ then $$v(x) = \lambda^\frac{2+\sigma}{p-1} u(\lambda x)$$ also solves the same equation. This immediately tells us that in the existence regime the equation \eqref{eqMAIN}$_\sigma$ in $n\geq 2$ has many solutions. However, in the case $n=1$ and in the existence regime (either $\sigma <-2$ and $p>-1-\sigma$ or $\sigma>-2$ and $p<-1-\sigma$), our special solution $C |x|^\frac{2+\sigma}{1-p}$ for some $C>0$ does not create any new solution because
$$
\lambda^\frac{2+\sigma}{p-1} \big| \lambda x \big|^\frac{2+\sigma}{1-p} = |x|^\frac{2+\sigma}{1-p}.
$$ 
Hence, the invariant property applied to the above special solution gives us nothing. Nevetheless, it seems that uniqueness of solutions does not occur. For example, in the special case $\sigma = 1$ and $p=-4$, the equation admits the following two solutions
\begin{equation}\label{eqExample}
C_1 |x|^{3/5}, \quad C_2 |x|^{3/5} (3|x|+5)^{2/5}.
\end{equation}
Since uniqueness is a delicated issue, we relegate this to section \ref{5NU}.

The rest of paper is organized as follows. In section \ref{sec-pre}, we prove some important properties of solutions to \eqref{eqMAIN}$_\sigma$. Then we establish the non-existence results for the case $n \geq 2$ and $p \leq 0$ in section \ref{sec-n>=2,p<0} and section \ref{sec-n>=2,p=0}. In section \ref{sec-n=2,0<p<=1}, we consider the case $n=2$ and $0<p \leq 1$. In section \ref{sec-n=1,p<=0} we consider the case $n=1$ and $p \leq 0$ and in section \ref{sec-n=1,p>0} we consider the case $n=1$ and $p>0$. Finally, section \ref{5NU} is devoted to uniqueness of solutions.

%\tableofcontents

%\vspace{-24pt

%%%%%%%%%
%%%%%%%%%
%%%%%%%%%

\section{Preliminaries}
\label{sec-pre}

In what follows, the notation $B_r$ is always understood as the open ball $B_r(0)$ in $\R^n$ centered at the origin with radius $r$. The boundary of $B_r$ is denoted by $\partial B_r$. When the function $u$ is radially symmetric with respect to the origin, instead of writing $u(x)$, we also use the notation $u(r)$ with $r$ being $|x|$. 

Throughout the paper and in the case $n \geq 2$, we use the notation $\overline u (r)$ with $r \geq 0$ to denote the spherical average of $u$ centered at the origin on the sphere $\partial B_r$, namely $\overline u (0) = u(0)$ and
\[
\overline u (r) = \frac{1}{|\partial B_r|}\int_{\partial B_r} u d\sigma
\]
if $r>0$. (For simplicity, we shall omit $d\sigma$.) Spherical averaging $\overline u$ has some nice properties such as it, as a function of $r \geq 0$, is continuous provided $u$ is continuous and it enjoys the following identity
\[
\overline{\Delta u} = \Delta \overline u
\]
provided $u$ is of class $C^2$. 

In the following subsections, we establish some nice properties for non-negative solutions to \eqref{eqMAIN}$_\sigma$ in $\R^n$. Note that, throughout the paper, the symbol $C$ denotes a generic positive constant whose value could be different from one line to another.

%%%%%%%%%
%%%%%%%%%

\subsection{The case $n\geq 2$}

First we start with the case $n \geq 2$, arbitrary $\sigma \in \R$, and arbitrary $p \ne 0$. Although it is possible to give a combined proof of Lemmas \ref{lem-basic-n>=2} and \ref{lem-basic-n=1} below, we prefer to highlight the key difference between the sign of solutions and keep them separate.

\begin{lemma}\label{lem-basic-n>=2}
Assume that $n \geq 2$ and that $u \geq 0$ is any non-trivial solution to \eqref{eqMAIN}$_\sigma$ for arbitrary $\sigma$ and arbitrary $p \ne 0$. Then we have the following claims
\begin{enumerate}
 \item the functions 
$$r \mapsto \overline u (r) \quad \text{and} \quad r \mapsto r^{n-1} \overline u'(r)$$ 
are decreasing in $(0,+\infty)$;
 \item there hold
\[
r^{n-1} \overline u'(r) \leq 0 \quad \text{and} \quad \overline u (r) \leq u (0) 
\]	
for all $r > 0$;
 \item there holds
\[
u(0)>0.
\]
\end{enumerate}
\end{lemma}

\begin{proof}
As $u \geq 0$ everywhere in $\R^n$, we claim that
$\Delta \overline u (r) = - \overline{|x|^\sigma u(r)^p } \leq 0$
in $[0,+\infty)$, which by the equality
\[
\Delta \overline u (r) = r^{1-n} \big(r^{n-1} \overline u'(r) \big)'
\]
for $r>0$ implies the decreasing of $r \mapsto r^{n-1} \overline u'(r)$ in $[0, +\infty)$. We note that we have not made any assumption on $p$ and $\sigma$.

To establish the monotonicity of $r\mapsto \overline u(r)$, we need more work and this requires $n \geq 2$. In view of the decreasing of $r \mapsto r^{n-1} \overline u'(r)$ there holds
\[
\lim_{r \searrow 0} r^{n-1}\overline{u}'(r) = L \in (-\infty, +\infty].
\]
Here $L$ could be $+\infty$. Our aim is to rule out the possibility of $L > 0$. 

\noindent\textbf{Case 1}. Suppose that $L \in (0, +\infty)$. Then for some $r_1 \in (0,1)$ we must have
\[
\overline{u}'(r) \geq \frac L2r^{1-n} 
\]
for all $r \in (0, r_1]$. Integrating both sides of the preceding inequality over $(r, r_1)$ gives
\[
\max_{\overline B_1} u \geq \overline{u}(r) \geq \overline{u}(r) - \overline{u}(r_1) \geq \frac L2 \int_r^{r_1} s^{1-n} ds
\]
for all $r \in (0, r_1]$. Keep in mind that $\max_{\overline B_1} u < +\infty$ thanks to the continuity of $u$. Hence, from the above argument and as $n \geq 2$ we easily obtain a contradiction if we let $r \searrow 0$ because the integral $\int_0^{r_1} s^{1-n} ds$ diverges.

\noindent\textbf{Case 2}. Suppose that $L = +\infty$. Then for some $r_1 \in (0,1)$ we must have
\[
\overline{u}'(r) \geq \frac 12r^{1-n} 
\]
for all $r \in (0, r_1]$. Repeating the above argument we also get a contradiction.

From the two cases above, we are in position to conclude that $L \leq 0$. This and the decreasing of $r \mapsto r^{n-1} \overline u'(r)$ in $[0, +\infty)$ imply that $\overline u'(r) \leq 0$ in $(0, +\infty)$. This tells us that the function $r \mapsto \overline u(r)$ is decreasing in $(0, +\infty)$.

Now we establish the inequality $\overline u (r) \leq u (0) $ for all $r \geq 0$. But this simply follows from the decreasing of $r \mapsto \overline u(r)$ in $(0, +\infty)$. Finally, one can easily conclude that $u(0) >0$. This is because otherwise the identity $u(0)=0$ and the decreasing of $r \mapsto \overline u(r)$ in $(0, +\infty)$ imply that $u$ must be trivial, which violates our hypothesis.
\end{proof}

\begin{remark}
We have the following two remarks:
\begin{enumerate}
 \item The condition $n \geq 2$ for the inequality $u(0) > 0$ in Lemma \ref{lem-basic-n>=2} is sharp in the sense that such an inequality is no longer true for $n=1$. This can be easily checked by seeing existence results in the case $n=1$; see Table \ref{table-n=1}. In sections \ref{sssec-n>=2,p<0,s>-2} and \ref{sssec-n>=2,p<0,s<=-2} below, we crucially use the condition $u(0)>0$.

 \item If $u > 0$ in $\R^n \setminus \{ 0 \}$, then actually one can show that the two functions $r \mapsto \overline u (r)$ and $r \mapsto r^{n-1} \overline u'(r)$ are strictly decreasing in $(0,+\infty)$. A typical example of this is when $u$ is a solution to \eqref{eqMAIN}$_\sigma$ with $p<0$.
\end{enumerate}
\end{remark}

Next we consider the case $p=0$. It is worth noting that in this particular case, we no longer require $u \geq 0$ in $\R^n$ as in Lemma \ref{lem-basic-n>=2} above.

\begin{lemma}\label{lem-basic-p=0}
Assume that $n \geq 2$, $\sigma \in \R$, and $u$ is any non-trivial solution to \eqref{eqMAIN}$_\sigma$ with $p=0$. Then we have the following claims
\begin{enumerate}
 \item the functions 
$$r \mapsto \overline u (r) \quad \text{and} \quad r \mapsto r^{n-1} \overline u'(r)$$ 
are strictly decreasing in $(0,+\infty)$;
 \item there hold
\[
r^{n-1} \overline u'(r) \leq 0 \quad \text{and} \quad \overline u (r) \leq u (0) 
\]	
for all $r > 0$;
 \item there holds
\[
u(0)>0.
\]
\end{enumerate}
\end{lemma}

\begin{proof}
The argument provided below is essentially the same as that used in the case $p \ne 0$. Indeed, by taking spherical averages centered at the origin, we obtain from the equation the following
\begin{equation}\label{eq3}
- (r^{n-1} \overline u' (r))' = r^{n-1+\sigma} >0
\end{equation}
for $r >0$. This in particular implies that the mapping $r \mapsto r^{n-1} \overline u'(r)$ is strictly decreasing in $(0, +\infty)$. As in the proof of Lemma \ref{lem-basic-n>=2}, we obtain
\[
\lim_{r \searrow 0} r^{n-1}\overline{u}'(r) = L \in (-\infty, +\infty]
\]
for some $L$ but it could be $+\infty$. Our aim is to rule out the possibility of $L > 0$. 

\noindent\textbf{Case 1}. Suppose that $L \in (0, +\infty)$. Then for some small $r_1 \in (0,1)$ we must have
\[
\overline{u}'(r) \geq \frac L2r^{1-n} 
\]
for all $r \in (0, r_1]$. Integrating both sides of the preceding inequality over $(r, r_1)$ gives
\[
2\max_{\overline B_1} |u| \geq \overline{u}(r) - \overline{u}(r_1) \geq \frac L2 \int_r^{r_1} s^{1-n} ds
\]
for all $r \in (0, r_1]$. Keep in mind that $\max_{\overline B_1} |u| < +\infty$ thanks to the continuity of $u$. Hence, from the above argument and as $n \geq 2$ we easily obtain a contradiction if we let $r \searrow 0$ because the integral $\int_0^{r_1} s^{1-n} ds$ diverges.

\noindent\textbf{Case 2}. Suppose that $L = +\infty$. Then for some $r_1 \in (0,1)$ we must have
\[
\overline{u}'(r) \geq \frac 12r^{1-n} 
\]
for all $r \in (0, r_1]$. Repeating the above argument we also get a contradiction.

It follows from the two cases above that $L \leq 0$. Keep in mind that the mapping $r \mapsto r^{n-1} \overline u'(r)$ is strictly decreasing in $(0, +\infty)$. Hence, it is easy to conclude that
\[
r^{n-1}\overline{u}'(r) < 0
\]
for all $r>0$. Hence $\overline{u}'(r) < 0$ for all $r>0$, giving the monotonicity of $\overline u$ as claimed. The remaining claims can be proved similarly.
\end{proof}

%%%%%%%%%
%%%%%%%%%

\subsection{The case $n=1$}

Let $u \geq 0$ be non-negative, non-trivial solution to the equation \eqref{eqMAIN}$_\sigma$ in $(0, +\infty)$, namely \eqref{eqMAIN-n=1}. An analog of Lemma \ref{lem-basic-n>=2} is the following.

\begin{lemma}\label{lem-basic-n=1}
Assume that $u \geq 0$ is any non-trivial, $C^2$-solution to \eqref{eqMAIN}$_\sigma$ in $(0, +\infty)$ with arbitrary $p, \sigma \in \R$, not necessarily continuous up to $0$. Then we have the following claims:
\begin{enumerate}
 \item Respectively, the functions 
$$r \mapsto u (r) \quad \text{ and } \quad r \mapsto u'(r)$$ 
are increasing and decreasing in $(0,+\infty)$;
 \item there exist three positive constants $c$, $C$, and $r_1$ such that
\[
c \leq u(r) \leq Cr
\]
for all $r \geq r_1$.
\end{enumerate}
\end{lemma}

\begin{proof}
As $u'' \leq 0$ in $(0, +\infty)$, we conclude that the function $r \mapsto u'(r)$ is decreasing in $(0, +\infty)$. This allows us to define
\[
\ell = \lim_{r \nearrow +\infty} u'(r) \in [-\infty, +\infty),
\]
here $\ell$ could be minus infinity. As $u \geq 0$ in $(0, +\infty)$ we necessarily have $\ell \geq 0$. Otherwise, $u' (r)$ becomes strictly negative for large $r$, which is not possible by seeing
\[
u(r) - u(r_0) = \int_{r_0}^r u'(s) ds \searrow -\infty.
\] 
Hence, there holds 
$$u' \geq 0 \quad \text{in }(0, +\infty),$$ 
namely $u$ is monotone increasing in $(0, +\infty)$. The monotonicity of $u$ tells us that there must exist some $c>0$ and some $r_2 > 0$ such that
\[
u(r) \geq c
\]
for all $r \geq r_2$. Otherwise, $u$ must vanish everywhere. To estimate $u$ from the above, we make use of the inequality $u'' \leq 0$. Indeed, this tells us that
\[
u'(r) \leq u'(r_2)
\]
and
\[
u(r) \leq u'(r_2) r + u(r_2) - u'(r_2) r_2
\]
for all $r \geq r_2$. In other words, $u$ grows at most linearly at infinity. Hence, there are $C>0$ and$r_1 \gg r_2$ such that
\[
u(r) \leq Cr
\]
for all $r \geq r_1$. This completes the proof.
\end{proof}

It is worth emphasizing that in Lemma \ref{lem-basic-n=1} we do not assume any condition on $\sigma$ and $p$. In addition, it is worth recalling that we do not assume any continuity up to the origin. We end this section by establishing a maximum principle type result.

\begin{lemma}\label{lem-MP-n=1}
Suppose that $u \geq 0$ is a non-trivial $C^2$-solution to the equation \eqref{eqMAIN}$_\sigma$ in $(0, +\infty)$. Then $u$ must be strictly positive everywhere.
\end{lemma}

\begin{proof}
Suppose that $u \geq 0$ is a non-trivial $C^2$-solution to the equation in $(0,+\infty)$. We claim that
\[
u(r) > 0 \quad \text{in } (0,+\infty).
\]
By convention, it suffices to consider the case $p \geq 0$. Indeed, by way of contradiction, assume that $u(2r_*)=0$ for some $r_* >0$. As $u$ is non-decreasing and $u \geq 0$, we must have $u(r) =0$ for all $0 < r \leq 2r_*$. In particular, we have
\[
0 = u'(r_*) \geq u'(r_*) - u'(r) = - \int_{r_*}^r u''(s) ds = \int_{r_*}^r s^\sigma u(s)^p ds 
\]
for all $r \geq r_*$. In the case $p>0$, as $u \geq 0$ is non-trivial, the integral $ \int_{r_*}^r s^\sigma u(s)^p ds >0$ for some large $r$. This is a contradiction. In the case $p=0$, the preceding estimate becomes
\[
0 = u'(r_*) \geq \int_{r_*}^r s^\sigma ds 
\]
for all $r \geq r_*$. Again, this is not true. Hence $u>0$ in $(0, +\infty)$ as claimed.
\end{proof}

We shall use the above maximum type principle in section \ref{sec-n=1,p>0} and section \ref{5NU}.
 
%%%%%%%%%
%%%%%%%%%
%%%%%%%%%

\section{The case $n \geq 2$}

This section is devoted to the proof of Theorem \ref{thm-MAIN}. As mentioned earlier, we are left with either $p \leq 0$ and $n \geq 2$ or $0<p \leq 1$ and $n=2$. For clarity, we divide our analysis into three parts as follows: in subsections \ref{sec-n>=2,p<0} and \ref{sec-n>=2,p=0}, we handle the case $p<0$ and $p=0$ respectively and in subsection \ref{sec-n=2,0<p<=1}, we treat the case $0<p \leq 1$ and $n=2$. 

By seeing Table \ref{table-1}, it is clear that we do not have any existence result in these three cases. Hence, we can argue by contradiction, namely throughout this section we always assume by way of contradiction that $u \geq 0$ is a solution to \eqref{eqMAIN}$_\sigma$ in $\R^n$ with $n \geq 2$.

\subsection{The case $n \geq 2$ and $p <0$}
\label{sec-n>=2,p<0}

By convention we have $u>0$ in $\R^n \setminus \{ 0\}$. Also, as the map $t \mapsto t^p $ is convex in $(0, +\infty)$ and one can apply Jensen's inequality to get
\[
\frac 1{|\partial B_r|}\int_{\partial B_r} |x|^\sigma u(x) ^p 
=\frac{r^\sigma}{|\partial B_r|} \int_{\partial B_r} u(x) ^p 
\geq r^\sigma \Big(\frac{1}{|\partial B_r|} \int_{\partial B_r} u(x) \Big) ^p 
\]
for $r>0$. Hence, by taking spherical averages our equation \eqref{eqMAIN}$_\sigma$ gives
\[
-\Delta \overline u (r) = \overline{|x|^\sigma u(r)^p } \geq r^\sigma \overline{u}(r)^p 
\]
for all $r > 0$, so that 
\[
-\Delta \overline u (r) \geq r^\sigma \overline{u}(r)^p \geq u(0)^p r^\sigma
\]
for all $r > 0$. Here, thanks to $p < 0$, we have made use of the inequality $\overline u (r) \leq u (0) $ established in Lemma \ref{lem-basic-n>=2}, to bound $\overline{u}(r)^p$ from the below. Thus, we have proved that
\begin{equation}\label{eq1}
	-\big(r^{n-1}\overline{u}'(r)\big)' \geq u(0)^p r^{n-1+\sigma} 
\end{equation}
for all $r >0$ and for all $p <0$. In fact, the estimate \eqref{eq1} still holds true when $p=0$. 

To tackle our equation, we consider the two cases $\sigma \geq -2$ and $\sigma <-2$ separately. Indeed, while in the case $\sigma > -2$, the asymptotic behavior of solutions far from the origin plays the essential role, we need to take care solutions near the origin in the case $\sigma \leq -2$. 

%%%%%%%%%
%%%%%%%%%

\subsubsection{The case $\sigma > -2$}
\label{sssec-n>=2,p<0,s>-2}

In this scenario, we have $n+\sigma >0$ and by Lemma \ref{lem-basic-n>=2} we know that $u(0)>0$ and $\overline{u}'(1) \leq 0$. Now integrating both sides of \eqref{eq1} over $[1,r]$ gives
	\begin{align*}
-r^{n-1}\overline{u}'(r) \geq \overline{u}'(1) -r^{n-1}\overline{u}'(r)
&\geq \frac{u(0)^p }{n+\sigma} \big( r^{n + \sigma} - 1\big)
	\end{align*}
for $r\geq 1$. Therefore,
\begin{equation}\label{for9}
	\overline{u}'(r) \leq - \frac{u(0)^p }{n+\sigma} r^{1 + \sigma} + \frac{u(0)^p }{n+\sigma} r^{1-n}
\end{equation}
for $r\geq 1$. Depending on the size of $n$, there are two cases:

\noindent\textbf{Case 1}. If $n > 2$ and $\sigma > -2$, then integrating \eqref{for9} over $[1,r]$ gives
	\begin{align*}
	\overline{u}(r)-\overline{u}(1) & \leq - \frac{u(0)^p }{n+\sigma} \frac{ r^{2+\sigma} - 1}{2+\sigma} + \frac{u(0)^p }{n+\sigma} \frac{r^{2-n} - 1}{2-n}	
	\end{align*}	
for $r\geq 1$. We have then $\overline{u}(r)\to -\infty$ as $r\to+\infty$ thanks to $u(0)>0$ and $n+\sigma \geq 2+\sigma>0$, which is contradiction with $u>0$ in $\R^n \setminus \{ 0\}$. 

\noindent\textbf{Case 1}. If $n =2$ and $\sigma > -2$, then integrating \eqref{for9} over $[1,r]$ gives
	\begin{align*}
	\overline{u}(r)-\overline{u}(1) & \leq - \frac{u(0)^p }{n+\sigma} \frac{ r^{2+\sigma} - 1}{2+\sigma} + \frac{u(0)^p }{n+\sigma} \log r
	\end{align*}	
for $r\geq 1$. Again, we still have $\overline{u}(r)\to -\infty$ as $r\to+\infty$ thanks to $u(0)>0$, $n+\sigma \geq 2+\sigma>0$, and $\log r =o( r^{2+\sigma})$ at infinity. This is still contradiction with $u>0$ in $\R^n \setminus \{ 0\}$. 

%%%%%%%%%
%%%%%%%%%
 
\subsubsection{The case $n \geq 2$ and $\sigma \leq -2$}
\label{sssec-n>=2,p<0,s<=-2}
 
Still by Lemma \ref{lem-basic-n>=2} we must have
\[
r^{n-1}\overline{u}'(r) \leq 0
\]
for all $r >0$. Integrating both sides of \eqref{eq1} over $(r_*, r)$ to get
\begin{align*}
- r^{n-1} \overline{u}'(r) \geq r_*^{n-1}\overline{u}'(r_*) - r^{n-1}\overline{u}'(r) \geq u(0)^p \int_{r_*}^r s^{n-1 + \sigma}ds.
\end{align*}
	
\noindent\textbf{Case 1}. Suppose $n+\sigma \leq 0$. We easily reach a contradiction if we fix $r>0$ and let $r_* \searrow 0$ because the integral $\int_0^r s^{n-1+\sigma} ds$ diverges. 

\noindent\textbf{Case 2}. Suppose $n+\sigma >0$. In this case, we should have
\begin{align*}
- r^{n-1} \overline{u}'(r) \geq \frac{u(0)^p }{n+\sigma} \big(r^{n + \sigma} - r_*^{n+\sigma} \big) 
\end{align*}
for all $r_* \in (0, r)$. Thanks to $n+\sigma>0$, by sending $r_* \searrow 0$ we get
\begin{align*}
- r^{n-1} \overline{u}'(r) \geq \frac{u(0)^p }{n+\sigma} r^{n + \sigma} 
\end{align*}
and this is true for all $r>0$. Hence, we have just shown that
	\begin{align*}
- \overline{u}'(r ) \geq \frac{u(0)^p }{n+\sigma} r^{1 + \sigma} 
\end{align*}
for all $r>0$. Keep in mind that $0 \leq \overline u(r) \leq u(0)$ for any $r \geq 0$, thanks to Lemma \ref{lem-basic-n>=2}. Integrating both sides over $(r, 1)$ gives
\begin{align*}
u(0) \geq \overline{u}(r ) - \overline{u}(1) \geq \frac{u(0)^p }{n+\sigma} \int_r^1 s^{1 + \sigma} ds 
\end{align*}
for all $r \in (0, 1)$. As $1+\sigma \leq -1$, $n+\sigma>0$, and $u(0)>0$, by letting $r \searrow 0$ we arrive at a contradiction because the integral $\int_0^1 s^{1+\sigma} ds$ diverges.

\begin{remark}
At first glance, one can notify that the condition $p < 0$ is not clearly mentioned in the proof above, however, it is worth noting that we actually make use of it because we have used \eqref{eq1} several times.
\end{remark}

%%%%%%%%%
%%%%%%%%%

\subsection{The case $n \geq 2$ and $p=0$}
\label{sec-n>=2,p=0}

In this particular case, we show that our equation \eqref{eqMAIN}$_\sigma$ in $\R^n$ with $p=0$ does not admit any solution $u \geq 0$. Note that in this case our equation becomes
\[
-\Delta u = |x|^\sigma \quad \text{ in } \R^n \setminus \{ 0\}
\]
The main difference between the two cases $p < 0$ and $p=0$ is that in the later case we no longer have the positivity of $u$ in $\R^n \setminus \{0\}$. Fortunately, thanks to the continuity of $u$, by replacing $u$ by $u + C$ for some sufficiently large constant $C>0$, if necessary, we can further assume that $u>0$ everywhere in $\overline B_1$. 

%%%%%%%%%
%%%%%%%%%

\subsubsection{The case $\sigma \leq -2$}

The argument provided below is essentially the same as that used in the case $p<0$. Our starting point is the following equation
\begin{equation}\label{eq3}
- (r^{n-1} \overline u' (r))' = r^{n-1+\sigma} >0
\end{equation}
for $r >0$, which is just \eqref{eq1} when $p=0$. Integrating both sides of \eqref{eq3} over $(r_*, r)$ with $0<r_*<r$ and making use of $r_*^{n-1}\overline{u}'(r_*) \leq 0$ to get
	\begin{align*}
- r^{n-1} \overline{u}'(r) \geq r_*^{n-1}\overline{u}'(r_*) - r^{n-1}\overline{u}'(r) \geq \int_{r_*}^r s^{n-1 + \sigma}ds.
	\end{align*}
Depending on the size of $n+\sigma$, we have two possible cases:
	
\noindent\textbf{Case 1}. Suppose $n+\sigma \leq 0$. Then we reach a contradiction if we fix $r>0$ and let $r_* \searrow 0$ because the integral $\int_0^r s^{n-1+\sigma} ds$ diverges. 

\noindent\textbf{Case 2}. Suppose $n+\sigma >0$. In this case, we should have
	\begin{align*}
- r^{n-1} \overline{u}'(r) \geq \frac{r^{n + \sigma} - r_*^{n+\sigma}}{n+\sigma} 
	\end{align*}
for all $r_* \in (0, r)$. Thanks to $n+\sigma>0$, by sending $r_* \searrow 0$ we get
	\begin{align*}
- r^{n-1} \overline{u}'(r) \geq \frac{r^{n + \sigma} }{n+\sigma} 
	\end{align*}
and in fact this is true for all $r>0$. Hence, we have just shown that
		\begin{align*}
- \overline{u}'(r ) \geq \frac{r^{1 + \sigma} }{n+\sigma} 
	\end{align*}
for all $r>0$. Now integrating both sides of the preceding inequality over $[r, 1]$ gives
\begin{align*}
2\max_{\overline B_1} |u| \geq \overline{u}(r ) - \overline{u}(1) \geq \frac{1 }{n+\sigma} \int_r^1 s^{1 + \sigma} ds 
\end{align*}
for all $r \in (0, 1)$. (It is not necessary to take absolute value of $u$ because we have already assumed at the beginning that $u>0$ in $\overline B_1$.) As $1+\sigma \leq -1$ and $n+\sigma>0$, by letting $r \searrow 0$ we arrive at a contradiction because the integral $\int_0^1 s^{1+\sigma} ds$ diverges. 

\begin{remark}
Notice that to derive contradiction in the case $\sigma \leq -2$ we only work near the origin. Therefore, we can further claim that in this particular case, namely $n \geq 2$ and $\sigma \leq -2$, our equation $-\Delta u = |x|^\sigma $ also does not admit any sign-changing solution.
\end{remark}

%%%%%%%%%
%%%%%%%%%

\subsubsection{The case $\sigma > -2$}

Recall from \eqref{eq3} that $\overline u$ solves
\[
-(r^{n-1} \overline u'(r))' = r^{n-1+\sigma} >0
\]
for $r>0$. Using this we can estimate
\[
 \overline u'(1) -r^{n-1} \overline u'(r) = \frac 1{n+\sigma} \big[ r^{n+\sigma} - 1 \big]
\]
for $r \geq 1$, which is equivalent to
\[
r^{1-n} \overline u'(1) - \overline u'(r) =- \frac 1{n+\sigma} r^{1-n} + \frac 1{n+\sigma} r^{1+\sigma} 
\]
for $r \geq 1$. Hence
\begin{equation}\label{eq-1}
\Big( \overline u'(1) + \frac 1{n+\sigma} \Big) \int_1^r s^{1-n} ds + \overline u(1) - \overline u(r) = \frac 1{(n+\sigma)(2+\sigma)} \big[ r^{2+\sigma} - 1 \big]
\end{equation}
for $r \geq 1$. Keep in mind that $\overline u \geq 0$ in $(0,+\infty)$.

\noindent\textbf{Case 1}. Suppose $n=2$. Then the estimate \eqref{eq-1} tells us that
\[
\Big( \overline u'(1) + \frac 1{2+\sigma} \Big) \log r+ \overline u(1) \geq \frac 1{(2+\sigma)^2} \big[ r^{2+\sigma} - 1 \big]
\]
for $r \geq 1$. From this and by sending $r \nearrow +\infty$ we obtain contradiction because $$\log r = o(r^{2+\sigma})$$ at infinity, thanks to $\sigma \geq -1$.

\noindent\textbf{Case 2}. Suppose $n \geq 3$. Then the estimate \eqref{eq-1} gives
\[
\Big( \overline u'(1) + \frac 1{2+\sigma} \Big) \frac{1 - r^{2-n} }{n-2}+ \overline u(1) \geq \frac 1{(2+\sigma)^2} \big[ r^{2+\sigma} - 1 \big]
\]
for $r \geq 1$. From this and by sending $r \nearrow +\infty$ we still obtain contradiction.

%%%%%%%%%
%%%%%%%%%

\subsection{The case $n =2$ and $0<p\leq 1$}
\label{sec-n=2,0<p<=1}

Let $u \geq 0$ solve \eqref{eqMAIN}$_\sigma$ in $\R^2$. We observe that in the case $\sigma \geq 0$ by definition any solution to \eqref{eqMAIN}$_\sigma$ in $\R^2$ is now of class $C^2(\R^2)$. Hence we simply make use of a Liouville type result for super-harmonic functions in two dimensions to obtain the non-existence result. Therefore, we limit ourselves to the case $\sigma<0$. For this reason, by making use of the maximum principle and as $u$ is non-trivial we deduce that $u>0$ in $\R^2 \setminus \{0\}$.

We note that in the case $0<p \leq 1$ the nonlinearity $u^p$ is concave, which makes the analysis through spherical averages more involved. To obtain the desired non-existence result, we follow the argument in \cite{SZ96, DQ20}. Depending on the size of $p$, we consider two cases: either $p \in (0,1)$ or $p=1$.

%%%%%%%%%
%%%%%%%%%

\subsubsection{The case $p=1$}

This part is easy to handle because after taking spherical averages centered at the origin, our equation gives
\[
-(r \overline u'(r))' = r^{1+\sigma} \overline u(r)
\]
in $(0,+\infty)$. Thanks to Lemma \ref{lem-basic-n>=2} we know that $r\overline u'(r) \leq 0$ for all $r>0$ and $\overline u$ is decreasing in $(0, +\infty)$. Hence, for arbitrary but small $\varepsilon \in (0,r)$, integrating the above differential inequality over $[\varepsilon,r]$ gives
\begin{align*}
-r\overline u'(r) \geq \varepsilon \overline u'(\varepsilon ) -r\overline u'(r)
& =- \int_\varepsilon^r (s\overline u'(s))'ds\\
&= \int_\varepsilon^r s^{1+\sigma} \overline u (s) ds
\geq \overline u \big( \frac r2 \big) \int_\varepsilon^{r/2} s^{1+\sigma} ds.
\end{align*}
Hence, if $\sigma \leq -2$, then we immediately obtain contradiction by letting $\varepsilon \searrow 0$ because the integral $ \int_0^{r/2} s^{1+\sigma} ds$ diverges. Now we consider the remaining case $\sigma \in (-2,0)$. In this case, we can let $\varepsilon \searrow 0$ to get
\[
-\overline u'(r) \geq r^{-1} \overline u \big( \frac r2 \big) \int_0^{r/2} s^{1+\sigma} ds
=\frac {r^{1+\sigma}}{(2+\sigma)2^{2+\sigma}} \overline u \big( \frac r2 \big) 
\] 
for any $r>0$. Integrating both sides over $[r,2r]$ with $r>0$ to obtain
\begin{align*}
\overline u(r) \geq \overline u(r) - \overline u(2r) &\geq \frac {1}{(2+\sigma)2^{2+\sigma}} \int_r^{2r} \overline u \big( \frac s2 \big) s^{1+\sigma} ds\\
&\geq \frac {2^{2+\sigma} -1 }{(2+\sigma)^2 2^{2+\sigma}} r^{2+\sigma} \overline u(r)
\end{align*}
for any $r>0$. From this we obtain contradiction by sending $r \nearrow +\infty$, thanks to the positivity of $\overline u$ and $\sigma>-2$.

%%%%%%%%%
%%%%%%%%%

\subsubsection{The case $0<p <1$}

For arbitrary $\delta \in (0,1)$ we consider the auxiliary function $u^\delta$. Clearly, $u^\delta$ solves
\[
-\Delta (u^\delta)=\delta (1-\delta) u^{\delta-2} |\nabla u|^2 + \delta |x|^\sigma u^{p-(1-\delta)}
\]
and following the argument in \cite{SZ96} we conclude that $$r \mapsto \overline{u^\delta}(r)$$ is monotone decreasing in $(0,+\infty)$. Now by taking spherical averages, there holds
\[
- \overline{u^\delta}'' - \frac 1r\overline{u^\delta}'=\delta (1-\delta) \overline{u^{\delta-2} |\nabla u|^2} + \delta r^\sigma \overline{u^{p-(1-\delta)}} 
\]
for all $r>0$. By H\"older's inequality, one has
\[
\overline{u^{\delta p}}^\frac{1}{\delta} 
\overline{u^\delta}^{1-\frac{1}{\delta}}
\leq \overline{u^{p-(1-\delta)}} ,
\]
which yields
\[
- \frac 1{\delta} \overline{u^\delta}^{\frac{1}{\delta}-1} \overline{u^\delta}'' - \frac 1r \big(\overline{u^\delta}^\frac{1}{\delta} \big)'
\geq (1-\delta) \overline{u^\delta}^{\frac{1}{\delta}-1} \overline{u^{\delta-2} |\nabla u|^2} + r^\sigma \overline{u^{\delta p}}^\frac{1}{\delta} 
\]
for all $r>0$. By Cauchy--Schwarz's inequality, we further obtain
\[
- \big(\overline{u^\delta}^\frac{1}{\delta} \big)'' - \frac 1r \big(\overline{u^\delta}^\frac{1}{\delta} \big)'
\geq r^\sigma \overline{u^{\delta p}}^\frac{1}{\delta} 
\]
for all $r>0$. Hence, by denoting 
\[
U = \overline{u^\delta}^\frac{1}{\delta}, \quad \phi =\overline{u^{\delta p}}^\frac{1}{\delta} ,
\]
we arrive at
\[
-U''(r) - \frac 1rU'(r) \geq r^\sigma \phi (r)
\]
for all $r>0$. Observe the following inequalities
\[
U'(r) = \frac 1{\delta} \overline{u^\delta}^{\frac{1}{\delta}-1}\big( \overline{u^\delta} \big)' \leq 0, \quad
\phi'(r) = \frac 1{\delta} \overline{u^{\delta p}}^{1-\frac{1}{\delta}} \big( \overline{u^{\delta p}} \big)' \leq 0.
\]
Hence, integrating over $[\varepsilon,r]$ gives
\begin{align*}
-U'(r) \geq r^{-1} \int_\varepsilon^r s^{1+\sigma} \phi (s) ds
\geq r^{-1} \phi \big( \frac r2 \big) \int_\varepsilon^{r/2} s^{1+\sigma} ds.
\end{align*}
As in the case $p=1$, if $\sigma \leq -2$, then we immediately obtain contradiction by letting $\varepsilon \searrow 0$. Now we consider the remaining case $\sigma \in (-2,0)$. In this case, we can let $\varepsilon \searrow 0$ to get
\begin{align*}
-U'(r) \geq r^{-1} \phi \big( \frac r2 \big) \int_0^{r/2} s^{1+\sigma} ds
=\frac {r^{1+\sigma}}{(2+\sigma)2^{2+\sigma}} \phi \big( \frac r2 \big) 
\end{align*}
for any $r>0$. Integrating both sides over $[r,2r]$ with $r>0$ to obtain
\begin{align*}
U(r) \geq U(r) - U(2r)& \geq \frac {1}{(2+\sigma)2^{2+\sigma}} \int_r^{2r} s^{1+\sigma}\phi \big( \frac s2 \big) ds\\
&\geq \frac {2^{2+\sigma} -1 }{(2+\sigma)^2 2^{2+\sigma}} r^{2+\sigma}\phi (r) = C r^{2+\sigma}\phi (r) 
\end{align*}
for any $r>0$. Hence, we have just shown that
\begin{align*}
 \overline{u^\delta} (r) \geq C^\delta r^{(2+\sigma)\delta} \overline{u^{\delta p}} (r)
\end{align*}
for any $r>0$. As $\overline u (r) \leq u(0)$ and by H\"older's inequality, we easily get
\[
\overline{u^\delta} (r) \leq \big( \overline u(r) \big)^{1-\frac{1-\delta}{1-\delta p}}
\big( \overline{u^{\delta p}}(r) \big)^{\frac{1-\delta}{1-\delta p}}
\leq u(0)^{1-\frac{1-\delta}{1-\delta p}}
\big( \overline{u^{\delta p}}(r) \big)^{\frac{1-\delta}{1-\delta p}}.
\]
This helps us to estimate
\begin{equation}\label{eq-Key}
\overline{u^\delta} (r) \leq u(0) C^{\delta - 1} r^{-\frac{1-\delta}{1-p}(2+\sigma)}
\end{equation}
for any $r>0$. We assume for a moment that $u$ enjoys the following simple comparison
\begin{equation}\label{eq-KeyGrow}
u(x) \geq \frac 1{|x|}
\end{equation}
for all $|x| \geq 1$. Having this in hand, we easily conclude that
\[
\overline{u^\delta} (r) \geq r^{-\delta}
\]
for all $r \geq 1$. Thus, to fulfill \eqref{eq-Key}, we must have
\[
\delta \geq \frac{1-\delta}{1-p}(2+\sigma).
\]
But this is impossible if we let $\delta \in (0,1)$ sufficiently small. Finally, we prove \eqref{eq-KeyGrow}. Let
\[
v(x) = \frac c{|x|}
\]
in $\R^2 \setminus \{0\}$ with $c = \min_{|x|=1} u(x) > 0$. This choice of $v$ guarantees
\[
u - v \geq 0
\]
on $\{ x : |x|=1\}$. In $\R^2$ we know that
\[
\Delta v = \frac 1{|x|^3} > 0
\]
everywhere in $\R^2 \setminus \{ 0\}$. As $u \geq 0$, we conclude that
\[
\liminf_{|x| \nearrow +\infty} (u - v)(x) \geq 0
\]
and 
\[
\Delta (u- v) \leq 0
\]
everywhere in the region $\{x : |x| > 1 \} $. Hence by the maximum principle, there holds
\[
u - v \geq 0
\]
everywhere in $\{x : |x| > 1 \} $, and this is exactly \eqref{eq-KeyGrow}.

\begin{remark}
The estimate \eqref{eq-KeyGrow} is valid for any super-harmonic functions in $\R^2$. In fact, one can replace $|x|^{-1}$ by any $|x|^{-\alpha}$ with $\alpha > 0$ because $\Delta (|x|^{-\alpha}) = \alpha^2 |x|^{-\alpha-2}$.
\end{remark}

%%%%%%%%%
%%%%%%%%%
%%%%%%%%%
%%%%%%%%%

\section{The case $n = 1$}

In this section, we prove Theorem \ref{thm-MAIN=n=1-R} and Theorem \ref{thm-MAIN=n=1}. To this purpose, we first spend sections \ref{sec-n=1,p<=0} and \ref{sec-n=1,p>0} to prove Theorem \ref{thm-MAIN=n=1}. Then in section \ref{ssect-thm-n=1-R} we show how to use Theorem \ref{thm-MAIN=n=1} to prove Theorem \ref{thm-MAIN=n=1-R}. 

We consider the two cases $p \leq 0$ and $p > 0$ separately.

\subsection{Proof of Theorem \ref{thm-MAIN=n=1}: the case $p \leq 0$}
\label{sec-n=1,p<=0}

By seeing Table \ref{table-n=1}, we need to consider the two cases: either $\sigma \leq -2$ or $\sigma \geq -1$ and $-1-\sigma \leq p \leq 0$.

%%%%%%%%%
%%%%%%%%%

\subsubsection{The case $\sigma \geq -1$ and $-1-\sigma \leq p \leq 0$}
\label{sec-s<=-1,-1-s<=p<=0}

Clearly $p+\sigma + 1 \geq 0$. We treat the two cases $p+\sigma + 1 > 0$ and $p+\sigma + 1 = 0$ separately.

\noindent\textbf{Case 1}. Suppose $p+\sigma + 1 > 0$. Recall from Lemma \ref{lem-basic-n=1} the estimate $u(r) \leq Cr$ for all $r \geq r_1>0$. This, the non-positivity of $p$, and the equation imply that
\[
-u''(r) = |r|^\sigma u(r)^p \geq C r^{p+\sigma}
\]
holds for all $r \geq r_1$. Keep in mind that $u' \geq 0$ in $(0, +\infty)$, which plays an important role. By integrating both sides of the preceding inequality over $[r_1, r]$, we easily obtain
\begin{equation}\label{eq-n=1,-1-s<=p<=0,s>=-1}
u'(r_1) \geq u'(r_1) -u'(r) \geq C \int_{r_1}^r s^{p+\sigma}ds = \frac C{p+\sigma+1} \big[ r^{p+\sigma+1} - r_1^{p+\sigma+1} \big]
\end{equation}
for all $r \geq r_1$. From this we obtain contradiction by letting $r \nearrow +\infty$ as the term $r^{p+\sigma+1}$ is unbounded, thanks to $p+\sigma+1>0$.

\noindent\textbf{Case 2}. Suppose $p+\sigma+1 = 0$. Then as above we can estimate
\[
-u''(r) = |r|^\sigma u(r)^{-1-\sigma} \geq C r^{-1}
\]
holds for all $r \geq r_1$. Integrating both sides over $[r_1, r]$ and still by the non-negativity of $u'$ we obtain
\[
u'(r_1) \geq u'(r_1) -u'(r) \geq C \int_{r_1}^r s^{-1} ds = C \big[ \log r - \log (r_1) \big].
\]
From this we obtain contradiction by letting $r \nearrow +\infty$ as the term $\log r$ is unbounded.

\begin{remark}
Obviously, the above argument does not work if $p<-1 - \sigma$ because in this case we always have $1+\sigma + p<0$. Hence the right hand side of \eqref{eq-n=1,-1-s<=p<=0,s>=-1} provides us nothing. In fact, our equation admits solutions in the case $p<-1-\sigma$ and $\sigma >-2$. Besides, the condition $\sigma \geq -1$ is implicitly used to guarantee the existence of $p \in [-1-\sigma, 0]$. 
\end{remark}

%%%%%%%%%
%%%%%%%%%

\subsubsection{The case $\sigma \leq -2$ and $p \leq 0$}
\label{sec-s<-2,p<=0}

Recall that $u$ is increasing in $(0,+\infty)$, hence $u^p$ is decreasing in $(0, +\infty)$. Therefore, from the equation we can estimate
\[
u'(r) \geq u'(r) - u'(1) = - \int_r^1 u''(s) ds = \int_r^1 s^\sigma u(s)^p ds 
\geq u(1)^p \int_r^1 s^\sigma ds 
\]
for all $r \in (0,1]$, which then implies
\[
u'(r) \geq \frac { u(1)^p }{1+\sigma} \big( 1 - r^{1+\sigma} \big)
\]
for all $r \in (0,1]$. Integrating both sides over $[r,1]$ gives
\begin{align*}
u(1) - u(r) &= \int_r^1 u'(s) ds \geq \frac { u(1)^p }{1+\sigma} \int_r^1 (1- s^{1+\sigma}) ds.
\end{align*}
Depending on the size of $\sigma$, we have two possible cases.

\noindent\textbf{Case 1}. Suppose $\sigma<-2$. In this can we compute the integral $\int_r^1$ to get
\begin{equation}\label{eq-int-1}
\int_r^1 (1- s^{1+\sigma}) ds
= \frac 1{2+\sigma}r^{2+\sigma} - r +\frac {1+\sigma}{2+\sigma} ,
\end{equation}
which then implies
\begin{align*}
-u(1) & \leq u(r) - u(1) 
\leq - \frac { u(1)^p }{(1+\sigma)(2+\sigma)}r^{2+\sigma} + \frac { u(1)^p }{1+\sigma} r - \frac { u(1)^p }{2+\sigma}
\end{align*}
Letting $r \searrow 0$ we obtain contradiction because $r^{2+\sigma} \nearrow +\infty$, thanks to $\sigma < -2$.

\noindent\textbf{Case 2}. Suppose $\sigma=-2$. In this case we notice that
\begin{equation}\label{eq-int-2}
\int_r^1 (1- s^{-1}) ds
=\log r - r + 1 ,
\end{equation}
which then implies
\begin{align*}
u(1) & \geq u(1) - u(r) \geq u(1)^p \big( -\log r + r - 1 \big).
\end{align*}
Letting $r \searrow 0$ we obtain contradiction because $\log r\searrow -\infty$.

%%%%%%%%%
%%%%%%%%%

\subsection{Proof of Theorem \ref{thm-MAIN=n=1}: the case $p>0$}
\label{sec-n=1,p>0}

Seeing Table \ref{table-n=1}, in this case, we only have existence result if either $\sigma < -2$ and $p>-1-\sigma$ or $\sigma \in (-2,-1)$ and $0<p<-1-\sigma$. Let $u$ be a non-negative $C^2$-solution. Thanks to Lemma \ref{lem-MP-n=1}, we know that $u>0$ everywhere in $(0, +\infty)$.

%%%%%%%%%
%%%%%%%%%

\subsubsection{The case $n=1$, $\sigma \leq -2$, and $0< p < 1$}
\label{sec-n=1,s<=-2,0<p<=1}

Thanks to Lemma \ref{lem-basic-n=1}, the function $u$ is increasing in $(0,+\infty)$, so is the function $u^p$ because $p>0$. In view of the estimate above, we easily get
\[
u'(r) \geq u'(r) - u'(1) = - \int_r^1 u''(s) ds = \int_r^1 s^\sigma u(s)^p ds 
\geq u(r)^p \int_r^1 s^\sigma ds 
\]
for all $r \in (0,1]$. As $u>0$ in $(0, +\infty)$, the above estimate leads us to
\[
u'(r) u(r)^{-p} \geq \frac { 1 - r^{1+\sigma} }{1+\sigma} \geq 0
\]
for all $r \in (0,1]$. Integrating both sides over $[r,1]$ gives
\begin{align*}
\frac 1{1-p} \big( u(1)^{1-p} - u(r)^{1-p} \big) &= \int_r^1 u'(s) u(s)^{-p} ds \geq \frac 1{1+\sigma} \int_r^1 (1- s^{1+\sigma}) ds 
\end{align*}
for all $r \in (0,1]$. Depending on the size of $\sigma$, we have two possible cases.

\noindent\textbf{Case 1}. Suppose $\sigma<-2$. In this case we can compute the integral $\int_r^1$ as in \eqref{eq-int-1} to get
\begin{align*}
-u(1)^{1-p}& \leq u(r)^{1-p} - u(1)^{1-p} 
\leq - \frac {1-p}{(1+\sigma)(2+\sigma)}r^{2+\sigma} + \frac {1-p}{1+\sigma} r - \frac {1-p}{2+\sigma}
\end{align*}
Letting $r \searrow 0$ we obtain contradiction because $r^{2+\sigma} \nearrow +\infty$, thanks to $\sigma < -2$.

\noindent\textbf{Case 2}. Suppose $\sigma=-2$. With help of \eqref{eq-int-2} we arrive at
\begin{align*}
-u(1)^{1-p}& \leq u(r)^{1-p} - u(1)^{1-p} \leq \log r - r + 1.
\end{align*}
Letting $r \searrow 0$ we obtain contradiction because $\log r\searrow -\infty$.

%%%%%%%%%
%%%%%%%%%

\subsubsection{The case $n=1$, $\sigma < -2$, and $1 \leq p \leq -1-\sigma$}
\label{sec-n=1,s<=-2,0<=p<=-1-s}

Assume that $u$ solves 
$$-u''(r) = r^\sigma u(r)^p \quad \text{in } (0,+\infty)$$ 
with $\sigma < -2$ and $1 \leq p \leq -1-\sigma$. By a simple calculation, it is not hard to verify that
\begin{equation}\label{eq-Kelvin}
- \Big(r u \big(\frac 1r \big) \Big)'' = - \frac 1{r^3} u'' \big(\frac 1r \big) = r^{-\sigma -3} u\big(\frac 1r \big)^p = r^{-p-\sigma -3} \Big(r u \big(\frac 1r \big) \Big)^p 
\end{equation}
for all $r>0$. Hence, the Kelvin type transformation $r \mapsto ru(1/r)$ solves 
\begin{equation}\label{eq-n=1,s<=-2,0<=p<=-1-s}
-u''(r) = r^{\widetilde \sigma} u^p
\end{equation}
in $(0, +\infty)$ with $\widetilde \sigma = -p-\sigma - 3 > -2$. By seeing Table \ref{table-n=1} and Remark \ref{rmk-PS12}, we quickly conclude that there is no solution to \eqref{eq-n=1,s<=-2,0<=p<=-1-s} in the case $\widetilde\sigma > -2$ and $p \geq 1$. This completes the present proof. 

\begin{remark}\label{rmk-n=1,s<=-2,0<=p<=-1-s}
It is worth noting that if we repeat the above argument but $p>-1-\sigma$ then we arrive at $\widetilde \sigma<-2$. Hence the transformed equation \eqref{eq-n=1,s<=-2,0<=p<=-1-s} and our equation are the same, which gives us nothing. In fact, in the case $\sigma < -2$ and $p>-1-\sigma$, there is one solution.
\end{remark}

%%%%%%%%%
%%%%%%%%%

In the next two subsections, we mainly concentrate on the case $-2 \leq \sigma <-1$. 

\subsubsection{The case $n=1$, $\sigma = -2$, and $p=1$}
\label{sec-n=1,s=-2,p=1}

In this special case, $u$ solves the Cauchy--Euler equation
$$ r^2 u''(r) + u(r) = 0 \quad \text{in } (0,+\infty).$$ 
The above equation can be solved analytically, and it is easy to verify that the general solution to the above Cauchy--Euler equation is of the form
\[
c_1 |x|^{1/2} \sin \Big( \frac{\sqrt 3}{2}\log |x| \Big) 
+c_2 |x|^{1/2} \cos \Big( \frac{\sqrt 3}{2}\log |x| \Big) ,
\]
where $c_1$ and $c_2$ are constants. Clearly, the general solution oscillates and this concludes that there is no non-negative solution $u$ to our PDE.

\begin{remark}
From the general solution to the above Cauchy--Euler equation, one can construct the following sign-changing solution, in the sense of \eqref{eqClassicalSolutionClass}, to our PDE
\[
u(x)=\left\{
\begin{aligned}
|x|^{1/2} \sin \Big( \frac{\sqrt 3}{2}\log |x| \Big) & & \text{if } x \ne 0,\\
0 & & \text{if } x=0,
\end{aligned}
\right.
\]
in the whole space $\R$. In other words, although our PDE does not have any non-negative solution, it does have sign-changing solutions.
\end{remark}

%%%%%%%%%
%%%%%%%%%

\subsubsection{The case $n=1$, $-2 < \sigma <-1$, and $-1-\sigma \leq p < 1$}
\label{sec-n=1,-2<s<-1,-1-s<p<1}

We argue as in section \ref{sec-n=1,s<=-2,0<=p<=-1-s} above. Indeed, as in \eqref{eq-Kelvin}, the Kelvin type transformation $r \mapsto ru(1/r)$ solves \eqref{eq-n=1,s<=-2,0<=p<=-1-s}, namely
\[
-u''(r) = r^{\widetilde \sigma} u^p
\]
in $(0, +\infty)$. In this case, there hold $\widetilde \sigma = -p-\sigma - 3 \leq -2$ and $p<1$. However, by seeing Table \ref{table-n=1}, we quickly conclude that the above admits no solution in the case $\widetilde\sigma \leq -2$ and $p<1$. This completes the proof. 

\begin{remark}
As in Remark \ref{rmk-n=1,s<=-2,0<=p<=-1-s} above, the duality argument does not work if $p \notin [-1-\sigma, 1)$.
\end{remark}

%\begin{remark}
%We should emphasize that by duality, the transformed function is not necessarily continuous up to the origin. Fortunately, our argument in the case $\sigma \leq -2$ and $p<1$ actually does not make use of this; see subsections \ref{sec-s<-2,p<=0} and \ref{sec-n=1,s<=-2,0<p<=1}.
%\end{remark}

%%%%%%%%%
%%%%%%%%%

In the last part, we focus on the case $\sigma \geq -1$. 

\subsubsection{The case $n=1$, $\sigma \geq -1$, and $p>0$}
\label{sec-n=1,-1<=s,0<p<1}

We now handle the case $\sigma \geq -1$. It is worth noting that our argument below works for any $p>0$, hence providing an alternative approach to the result in \cite{PS12}. Recall from Lemma \ref{lem-basic-n=1} the estimate $u(r) \geq c$ for all $r \geq r_1$. Keep in mind that $p > 0$. Hence it follows from the equation that
\[
-u''(r) = r^\sigma u(r)^p \geq c^p r^\sigma
\]
holds for all $r \geq r_1$. By integrating both sides over $[r_1, r]$, we easily obtain
\[
u'(r_1) -u'(r) = -\int_{r_1}^r u''(s) ds \geq c^p \int_{r_1}^r s^\sigma ds
\]
for all $r \geq r_1$. We have two cases.

\noindent\textbf{Case 1}. Suppose $\sigma>-1$. Keep in mind that $u' \geq 0$ in $(0,+\infty)$. Hence the above calculation becomes
\begin{equation*}\label{eq-n=1,p<=0}
u'(r_1) \geq u'(r_1) -u'(r) = -\int_{r_1}^r u''(s) ds \geq \frac {c^p} {\sigma+1} \big[ r^{\sigma+1} - r_1^{\sigma+1} \big]
\end{equation*}
for all $r \geq r_1$. From this we obtain contradiction by letting $r \nearrow +\infty$ as the term $r^{\sigma+1}$ is unbounded, thanks to $\sigma+1>0$.

\noindent\textbf{Case 2}. Suppose $\sigma=-1$. Then, the above calculation becomes
\begin{equation*}\label{eq-n=1,p<=0}
u'(r_1) \geq u'(r_1) -u'(r) = -\int_{r_1}^r u''(s) ds \geq c^p \big[ \log r - \log (r_1) \big]
\end{equation*}
for all $r \geq r_1$. From this we obtain contradiction by letting $r \nearrow +\infty$ as the term $\log r$ is also unbounded.

%%%%%%%%%
%%%%%%%%%

\subsection{Proof of Theorem \ref{thm-MAIN=n=1-R}}
\label{ssect-thm-n=1-R}

We now prove Theorem \ref{thm-MAIN=n=1-R}. As any solution to the equation \eqref{eqMAIN-n=1-R} in $\R$ is also a solution to \eqref{eqMAIN-n=1} in $(0,+\infty)$. Hence, non-existence results in Theorem \ref{thm-MAIN=n=1-R} follows from non-existence results in Theorem \ref{thm-MAIN=n=1}. Besides, existence results in Theorem \ref{thm-MAIN=n=1-R} also follows from existence results in Theorem \ref{thm-MAIN=n=1}. Hence, we are left with the case $\sigma \geq 0$ and $p<-1-\sigma$, and seeing Table \ref{table-n=1-R}, we have a non-existence result in this regime. In fact, the argument below works for any $p \in \R$. By way of contradiction, assume that a solution $u \geq 0$ to $-u''(x) = |x|^\sigma u(x)^p$ in $\R$, in the sense of \eqref{eqClassicalSolutionClass}, exists. As $\sigma \geq 0$, $u$ is of class $C^2(\R)$. Hence $u$ is concave and bounded from below by $0$. It is elementary to verify that $u$ must be constant which must also be zero. In other words, we obtain a non-existence result in the regime $\sigma \geq 0$.

\begin{remark}
Although the function $C|x|^\frac{2+\sigma}{1-p}$ verifies \eqref{eqMAIN-n=1-R} in the regime $\sigma > -2$ and $p<-1-\sigma$, it is not a solution to \eqref{eqMAIN-n=1-R} in the regime $\sigma \geq 0$ and $p<-1-\sigma$ because it is not $C^1$ at the origin, thanks to $(2+\sigma)/(1-p) \in (0,1)$.
\end{remark}

%%%%%%%%%
%%%%%%%%%

% 原稿 1/16
%%%%%%%%%%%%%%%%%%%%%%%%%%%%%%%%%%%%%%%%%%%%%
\section{Non uniqueness of a positive solution when $n=1$} \label{5NU}

In this section, we are interested in the uniqueness issue of a non-negative, non-trivial solution to the one-dimensional equation
\begin{equation} \label{5E1}
	u''(x) + x^\sigma u^p (x) = 0 \quad \text{for}\ x > 0.
\end{equation}
Thanks to Lemma \ref{lem-MP-n=1}, any non-negative solution to \eqref{5E1} is indeed strictly positive. Hence, from now on, by a solution to \eqref{5E1} we mean a strictly positive $C^2$-function in $(0, +\infty)$. Seeing Theorem \ref{thm-MAIN=n=1}, we shall discuss the following two cases
\begin{enumerate} % ラベルは先頭のみ
[label=(\Alph*)]
\item \label{5C1} $\sigma>-2$ and $p<-1-\sigma$; 
\item \label{5C2} $\sigma<-2$ and $p>-1-\sigma$.
\end{enumerate}
As already observed in the Introduction, for $p \ne 1$ and $\sigma \ne -2$, there is a solution $u_a$ to \eqref{5E1} of the form
\[
u_a(x) = c_a x^a
\]
with
\begin{align*}
	a = \frac{\sigma+2}{1-p} \quad
	c_a &:= \left(a(1-a)\right)^{1/(p-1)}.
\end{align*}
The two conditions \ref{5C1} and \ref{5C2} are nothing but to guarantee $a \in (0,1)$. Apparently, the solution $u_a$ is continuous up to $x=0$. Hence, a natural question is that \textit{given $a \in (0,1)$, is there any solution $u$ to \eqref{5E1} in $(0,+\infty)$ rather than $u_a$, which is continuous up to $x=0$}? 

By a simple phase plane analysis, see section \ref{5NU6}, it turns out that the equation \eqref{5E1} admits other solutions rather than $u_a$. More precisely, for the case \ref{5C1}, our non-uniqueness result reads as follows.

\begin{theorem} \label{5T1}
Assume \ref{5C1}.
 There exists a one-parameter family of positive, classical solutions to \eqref{5E1}, denoted by $\{u^\alpha\}_{\alpha>0}$, such that
\begin{align*} 
\lim_{x\searrow 0} \frac{u^\alpha(x)}{ u_a(x)} = 1
\end{align*}
and if $\alpha_1 > \alpha_2 > 0$, then
\begin{align*} 
u^{\alpha_1}(x) > u^{\alpha_2}(x) > u_a(x) 
	\quad \text{for any}\ x > 0.
\end{align*}
\end{theorem}

We prove Theorem \ref{5T1} in section \ref{5NU3} below. The role of the parameter $\alpha$ is described by \eqref{5E13} below. It turns out that all solutions to the equation is above $u_a$; see section \ref{5NU4}. For the case \ref{5C2}, our non-uniqueness result reads as follows.

\begin{theorem} \label{5T2}
Assume \ref{5C2}. We have the following claims:
\begin{enumerate}
\item If $a>1/2$, then there is a positive solution $u$ of \eqref{5E1} satisfying
\begin{align*}
\lim_{x\searrow0} \frac{u(x)}{u_a(x)} = 0 
\quad\text{and}\quad
\lim_{x\nearrow+\infty} \frac{u(x)}{u_a(x)} = 1.
\end{align*}

\item If $a=1/2$, then there is a positive solution $u$ of \eqref{5E1} satisfying
\begin{align*}
\lim_{x\searrow0}\frac{u(x)}{u_a(x)} = 0 
\quad\text{and}\quad
\lim_{x\nearrow+\infty}\frac{u(x)}{u_a(x)} = 0.
\end{align*}

\item If $a<1/2$, then there is a positive solution $u$ of \eqref{5E1} satisfying
\begin{align*}
\lim_{x\searrow0}\frac{u(x)}{u_a(x)} = 1
\quad\text{and}\quad
\lim_{x\nearrow+\infty}\frac{u(x)}{u_a(x)} = 0.
\end{align*}
\end{enumerate} 
In addition to the existence, the solution is unique up to scaling transformation $u\mapsto u_\lambda$ with $u_\lambda(x)=\lambda^{-a}u(\lambda x)$.
\end{theorem}

We prove Theorem \ref{5T2} in section \ref{5NU5} below. Putting the two theorems above together, we deduce that for both cases \ref{5C1} and \ref{5C2}, uniqueness for positive solutions to \eqref{5E1} fails.

\subsection{Equivalent equations} \label{5NU1} % Subsection 1.1

Let $u$ be a positive solution of \eqref{5E1}. We set $$u=u_a v.$$ Then it is easy to verify that $v$ is positive and is of class $C^2$. Moreover, $v$ solves
\[
v'' + 2ax^{-1} v' + c^{p-1}_a x^{-2} v(v^{p-1}-1) = 0 \quad \text{for}\ x > 0.
\]
A more convenient form of the above equation is
\begin{equation} \label{5E3A}
	x^2 v'' + 2axv' + c^{p-1}_a v(v^{p-1}-1) = 0
\quad \text{for}\ x > 0	
\end{equation}
or
\begin{equation} \label{5E3}
	 x(xv') ' + (2a-1)xv' + c^{p-1}_a v(v^{p-1}-1) = 0
	\quad \text{for}\ x > 0.
\end{equation}
It is well-known that \eqref{5E3} can be written in the form of first order system
\begin{align}
\left\{
\begin{aligned} \label{5E4}
	& xy'_1 = y_2, \\
	& xy'_2 = -(2a-1)y_2 - c^{p-1}_a y_1(y^{p-1}_1-1)
\end{aligned}
\right.
\end{align}
for $x > 0$ by setting $y_1=v$ and $y_2=xv'$. The system \eqref{5E4} is often called a Briot--Bouquet system and it is well studied especially when the nonlinear term is analytic; see e.g.\ \cite{GT}. We next change the independent variable $x$ to $z=\log x$. Since $xd/dx=d/dz$, the equation \eqref{5E3} becomes
\begin{equation} \label{5E5}
	\ddot{V} + (2a-1) \dot{V} + c^{p-1}_a V(V^{p-1}-1) = 0
	\quad \text{for}\ z \in \R
\end{equation}
for $V(z)=v(e^z)$, where $\dot{V}=dV/dz$ and $\ddot{V}=d^2V/dz^2$. This is now an autonomous differential equation. The behavior of $v$ near $x=0$ for \eqref{5E3} corresponds to that of $V$ near $z=-\infty$.

%%%%%%%%%%%%%%%%%%%%%%%%%%%%%%%%%%%%%%%%%%%%%
\subsection{A local solution to a Briot--Bouquet system} \label{5NU2} % Subsection 1.2

We consider a Briot--Bouquet system
\begin{equation} \label{5E6}
	x \frac{dy}{dx} = Ay + f(y)
	\quad \text{for}\ x > 0.
\end{equation}
Here $A$ is a real $m\times m$ matrix and $f$ is an $\R^m$-valued $C^1$-function in a neighborhood of zero of $\R^m$ satisfying $$\lim_{y\to0} \frac {f(y)}y=0;$$ the unknown function $y$ is $\R^m$-valued. We start with the existence of a local solution to \eqref{5E6} with prescribed asymptotic behavior near zero.
\begin{lemma} \label{5T3}
Let $\gamma$ be a positive eigenvalue of $A$ and $y_0$ be its eigenvector.
 Then, there exists a (unique) local solution $y\in C^1(0,T)$ to \eqref{5E6} for some $T>0$ such that
\[
	\lim_{x\searrow 0} \frac{y(x)}{x^\gamma} = y_0.
\]
\end{lemma}
\begin{proof}
Since $y'=\gamma x^{\gamma-1}\varphi+x^\gamma\varphi'$ for $y=x^\gamma\varphi$, the equation \eqref{5E6} is equivalent to the equation of $\varphi$ of the form
\begin{equation} \label{5E7}
	x\varphi' = (A-\gamma)\varphi + x^{-\gamma} f(x^\gamma\varphi)
	\quad \text{for}\ x > 0.
\end{equation}
It suffices to find a solution $\varphi\in C[0,T)\cap C^1(0,T)$ for \eqref{5E7} for some $T>0$ such that $\varphi(0)=y_0$. As $A y_0 = \gamma y_0$, we see that $\psi=\varphi-y_0$ solves
\[
	x\psi' = (A-\gamma)\psi + x^{-\gamma} f\left(x^\gamma(y_0+\psi)\right)
\]
or
\[
(x\psi)' = F(x,\psi)
\]
with
\[	
	F(x,\psi) = (A-\gamma+1)\psi + x^{-\gamma} f\left(x^\gamma(y_0+\psi)\right).
\]
%since $(x\psi)'=x\psi'+\psi$.
 Our problem \eqref{5E7} with $\varphi(0)=y_0$ is 
% 原稿 4/16
equivalent to the following integral equation
\begin{equation} \label{5E8}
	\psi(x) = \frac{1}{x} \int^x_0 F\left(s,\psi(s)\right)ds
	\quad \text{with}\ \psi(0) = 0,
\end{equation}
where $\psi\in C^1(0,T)\cap C^0[0,T)$ for small $T>0$. Since $\gamma>0$ and $f(y)/y\to0$ as $y\to0$, setting $F(0,0)=0$ yields that $F\in C^1\left([0,\delta)\times(-\delta,\delta)^m\right)$ for small $\delta>0$ provided that $y_0$ is taken sufficiently small. Since $f(y)/y\to0$ as $y\to0$, we observe that $(\partial F/\partial\psi)(0,0)=0$. We now apply Lemma \ref{5L8} below to conclude that there is a unique solution to \eqref{5E8} for small $T>0$. The proof is now complete.
\end{proof}

\begin{lemma}[see \cite{GG}] \label{5L8}
For $F\in C^1\left([0,\delta)\times(-\delta,\delta)^m\right)$, let $w_0\in\R^m$ satisfy
$$F(0,w_0)=w_0, \quad \frac{\partial F}{\partial w}(0,w_0)=0.$$
 Then, there exists a unique function $w\in C[0,T)\cap C^1(0,T)$ for small $T>0$, which solves
\[
	w(x) = \frac{1}{x} \int^x_0 F\left(s,w(s)\right)ds,
	\quad w(0) = w_0 \quad x\in(0,T).
\] 
\end{lemma}
This can be proved by observing that
$w \mapsto (1/x) \int^x_0 F\left(s,w(s)\right)ds$
is a contraction mapping in $X=\left\{w\in C[0,T]\mid w(0)=w_0 \right\}$ if $T$ is sufficiently small; see \cite[Lemma 8]{GG}.

As an application of Lemma \ref{5T3}, we obtain the following result which allows us to construct local solutions to \eqref{5E3}.

\begin{corollary} \label{5T6}
Let $\mu_+$ be a positive root of 
\begin{equation} \label{5E9}
	\mu^2 + b\mu + c = 0.
\end{equation}
Let $f_0$ be a $C^1$ function near $0$ such that $f(0)=f'(0)=0$. Given any $w_0\in\R$, there exist some $T>0$ and a unique local solution $w\in C[0,T]\cap C^2(0,T)$ to
\begin{equation*} \label{5E10}
x(xw')' + bxw' + cw + f_0(w) = 0\quad \text{for}\ x \in (0,T)
\end{equation*}
such that
\begin{equation} \label{5E11}
	\lim_{x\searrow 0} \frac{w(x)}{x^{\mu_+}} = w_0.
\end{equation}
In addition, $w$ is non-decreasing near $x=0$ if $w_0>0$ and is non-increasing near $x=0$ if $w_0<0$.
\end{corollary}
\begin{proof} 
We set $y_1=w$, $y_2=xw'$ to get an equation for $y=(y_1,y_2)$ of the form
\[
	x \frac{dy}{dx} = Ay + f(y)
\]
with
\[
	A = \begin{pmatrix}
	0 & 1 \\
	-c & -b \\
\end{pmatrix}, 
\quad
	f(y) = \begin{pmatrix}
	0 \\
	-f_0( y_1 ) \\
\end{pmatrix}.
\]
The eigen equation for $A$ is nothing but \eqref{5E9}.
 Since $y_0=(w_0,\mu_+w_0)^T$ is an eigenvector for the eigenvalue $\mu_+$, applying Lemma \ref{5T3} yields the desired solution $w$ satisfying \eqref{5E11} and
\[
\lim_{x\searrow 0} \frac{w'(x)}{x^{\mu_+ - 1}} = \mu_+ w_0 .
\]
The above limit implies the monotonicity of $w$ near $x=0$ as stated.
\end{proof}

In the next two sections, we prove Theorem \ref{5T1} and Theorem \ref{5T2}.

%%%%%%
%%%%%%

\subsection{The case \ref{5C1}: proof of Theorem \ref{5T1}} \label{5NU3} % Subsection 1.3

Let $\alpha\geq0$ be arbitrary but fixed. We note that $v\equiv1$ is a stationary solution to \eqref{5E3}.
 We then set $v=1+w$ to obtain
\begin{equation} \label{5E12}
	x(xw')' + (2a-1)xw' + c^{p-1}_a(p-1)w + f_0(w) = 0,
\end{equation}
where 
$$f_0 (w) = a(1-a) \big[ (1+w)^p - 1 - pw \big].$$
Obviously, $f_0$ is smooth near $w=0$ and $f_0(w)/w\to0$ as $w\to0$. Our aim is to show that \eqref{5E12} admits a solution. This is done through several steps.

First we show that \eqref{5E12} admits a local solution near $x=0$. By \ref{5C1}, we observe that $p-1<0$. Thus, the equation \eqref{5E9} with $b=2a-1$ and $c=c^{p-1}_a(p-1)$ has always two real roots $\mu_+$ and $\mu_-$ satisfying $\mu_+>0>\mu_-$. Then for the rate $w_0 := \alpha$, by Corollary \ref{5T6}, there exists a local solution $w$ to \eqref{5E12} satisfying the asymptotic behavior \eqref{5E11}. We set 
\[
\eta^\alpha(x) = 
\left\{\
\begin{aligned}
& x^{-\mu_+} w(x) & &\text{for } x>0,\\
& \alpha & & \text{for } x= 0.
\end{aligned}
\right.
\]
By \eqref{5E11}, the function $\eta$ is continuous up to $x=0$. Our desired solution $u^\alpha$ to \eqref{5E1} is defined as follows
\begin{equation} \label{5E13}
	u^\alpha(x) = u_a(x) \left(1 + \eta^\alpha(x) x^{\mu_+}\right)
\end{equation}
at least near $x=0$. By definition, $u^0 \equiv u_a$ and $u^0 \not\equiv u_a$ if $\alpha>0$.

We next observe that the class $\{u^\alpha\}_{\alpha \geq 0}$ enjoys the order preserving property in the following sense
\begin{equation} \label{5E14}
	u^{\alpha_1}(x) > u^{\alpha_2}(x)
	\quad\text{for}\ x \in (0,T_*]\ \text{ if}\ \alpha_1 > \alpha_2 \geq 0
\end{equation}
as far as $u^{\alpha_1}$ and $u^{\alpha_2}$ exist on $(0,T_*]$. Admitting this fact and as $u^0 \equiv u_a$, we conclude that $u^{\alpha_2}>u_a$ as far as $u^{\alpha_2}$ exists so it can be extended to global positive solution; in addition, no blow up occurs since $u^{\alpha_2}$ is always concave.

It remains to prove the order preserving property \eqref{5E14}. This can be proved by the maximum principle. This order preserving property holds at least near $x=0$ by the form \eqref{5E13} of $u^\alpha$.
 Suppose \eqref{5E14} were false.
 Then, there would exist $x_0>0$ such that
\begin{align*}
u^{\alpha_1}(x_0) = u^{\alpha_2}(x_0) \quad \text{and} \quad 
u^{\alpha_1}(x) > u^{\alpha_2}(x)\quad \text{for}\ x \in (0,x_0).
\end{align*}
We consider equation \eqref{5E3A} for $v^\alpha$, namely
\[
x^2 ( v^\alpha )'' + 2ax (v^\alpha)' + c^{p-1}_a \big( (v^\alpha)^p-v^\alpha \big) = 0.
\]
Subtracting equation for $v^{\alpha_2}$ from that for $v^{\alpha_1}$, we observe that $h=v^{\alpha_1}-v^{\alpha_2}$ solves
\[
	x^2 h'' + 2axh' + c^{p-1}_a \beta(x)h = 0
\]
with
\[
\beta(x) = \int^1_0 g' \left(\theta v^{\alpha_2}(x) + (1-\theta)v^{\alpha_1}(x)\right)d\theta
	\quad\text{and} \quad 
g(v) = v^p-v.
\]
% 原稿 7/16
As $v^{\alpha_1}, v^{\alpha_2} \geq 1$, we must have $\theta v^{\alpha_2} + (1-\theta)v^{\alpha_1} \geq 1$ for any $\theta \in [0,1]$. Since $p<1$, we see that
\[
	g'(v) = pv^{p-1}-1 < 0 \quad \text{for}\ v > 1
\]
so that $\beta(x)<0$ for $x\in(0,x_0)$. Since $h(0)=h(x_0)=0$ and $h(x)>0$ for $x\in(0,x_0)$, there must be a positive maximum at some $x_*\in(0,x_0)$. By the equation for $h$, we must have $h''(x_*)>0$. However, this would contradict the maximum principle $h''(x_*)\leq0$. We now conclude \eqref{5E14}. 

Thus, the proof of Theorem \ref{5T1} is now complete.

\begin{remark}
Part of the order preserving property \eqref{5E14} indicates that the specific solution $u_a$ serves as a lower bound for the one-parameter family $\{u^\alpha\}_{\alpha \geq 0}$. Remarkably, in Theorem \ref{5T8} below, we show that this is in fact true for all solutions under the case \ref{5C1}. However, this is no longer true in the case \ref{5C2} as Theorem \ref{5T2-new} shows.
\end{remark}

% 原稿 10/16
%%%%%%%%%%%%%%%%%%%%%%%%%%%%%%%%%%%%%%%%%%%%
\subsection{The case \ref{5C2}: proof of Theorem \ref{5T2}} \label{5NU5} % Subsection 1.5

Recall the decomposition $u=u_a v$ and $V(z) = v(e^z)$. We also recall from \eqref{5E5} that $V$ solves the equation
\[
	\ddot{V} + (2a-1)\dot{V} + c^{p-1}_a V(V^{p-1}-V) = 0
\]
with $c^{p-1}_a=a(1-a)$. The above equation is a special example of Lienard system and there is a rich literature; see e.g.\ \cite[Section 3.8]{Pe}. Our aim is still to prove that the preceding equation, namely \eqref{5E5}, admit a global solution $V$. Clearly, one can rewrite \eqref{5E5} as
\[
\frac d{dz} \begin{pmatrix}
V\\
\dot V
\end{pmatrix}
=
- \begin{pmatrix}
\dot V\\
(2a-1) \dot{V} + c^{p-1}_a (V^p-V)
\end{pmatrix}.
\] 
From this one can easily check that there are two equilibria $(0,0)$ and $(1,0)$. The linerization of the preceding system is as follows
\[
\frac d{dz} \begin{pmatrix}
V\\
\dot V
\end{pmatrix}
=
- \begin{pmatrix}
0 & 1\\
c^{p-1}_a (p V^{p-1}-1) & 2a-1
\end{pmatrix}
\begin{pmatrix}
V\\
\dot V
\end{pmatrix}.
\] 
Keep in mind that $c^{p-1}_a = a(1-a)$. Hence, the solution behaves like
\[
	\ddot{V} + (2a-1)\dot{V} + a(1-a)(p-1)V = 0
\]
near $V=1$ and like
\[
	\ddot{V} + (2a-1)\dot{V} - a(1-a)V = 0
\]
near $V=0$. Since the case $a\leq1/2$ can be obtained by transforming $z\mapsto -z$, we may assume that $a\geq1/2$. In the sequel, we consider the two cases $a>1/2$ and $a=1/2$ separately.
 
%%%%%%%
%%%%%%%
 
\subsubsection{The case $a > 1/2$}

In this case, it suffices to construct a heteroclinic orbit satisfying $V\geq0$. Because the linearized equation around $V=0$ is of the form
\[
	\ddot{V} + (2a-1)\dot{V} - a(1-a)V = 0,
\]
it is easy to verify that its characteristic roots are $-a$ and $1-a$ with $1-a>0$. By local existence result established in section \ref{5NU2}, there is a solution to
\[
	x(xv')' + (2a-1)xv' + c^{p-1}_a v(v^{p-1}) = 0
\]
for $x>0$ with the following asymptotic behavior
\[
\lim_{x \searrow 0} \frac{v(x)}{x^{1-a}}=v_0
\]
for any given $v_0>0$. This guarantees that there is a solution to \eqref{5E5} for $z<z_0$ with some $z_0\in\R$ such that
\[
	\lim_{z\searrow-\infty} \frac{V(z)}{e^{(1-a)z}} = v_0.
\]
In particular, $V \searrow 0$ as $z \searrow -\infty$. Note that $V$ is uniquely determined by $v_0>0$. (The local existence result to \eqref{5E1} in section \ref{5NU2} corresponds to the existence of a local unstable manifold \cite{Pe} for \eqref{5E5}.) To be a heteroclinic orbit, it remains to prove that this $V$ can be extended globally and tends to $1$ as $z\nearrow+\infty$. Let $F$ be a primitive of $c^{p-1}_a v(v^{p-1}-1)$, i.e.
\[
F(w) = \int^w_1 c^{p-1}_a v(v^{p-1}-1) dv
=\frac{c^{p-1}_a}{2(p+1)} \big( 2 w^{p+1} - (p+1)w^2 + p-1 \big).
\]
For $w \geq 0$, this $F$ has the global minimum $0$ at $w=1$. In particular, $F \geq 0$. Now multiply $\dot{V}$ with \eqref{5E5} to get
\[
	\frac{d}{dz} E(V, \dot{V}) = -(2a-1) \dot{V}^2
\]
with the total energy
\[
	E(V, \dot{V}) = \frac{1}{2} |\dot{V}|^2 + F(V).
\]
For simplicity, we denote $Y=(V, \dot{V})$. Obviously, $E(Y) \geq 0$ and equality occurs at $Y=(1,0)$. Since $a>1/2$, the energy $E$ must decrease along $Y$. Note that as $\partial E/\partial V$ does not depend on $\dot V$ each level curve of $E$ is transversal to $V$-axis. Consequently, $E(Y)$ is strictly decreasing along $Y$. Thus the solution $Y$ must converge to the minimizer of $E(Y)$, namely
$$Y(z)\to(1,0) \quad \text{as } z \nearrow +\infty.$$
Moreover, $V>0$ since $E(Y)\leq E(0,0)$. We have thus constructed a positive solution $V$ of \eqref{5E5} such that
\[
	\lim_{z\searrow-\infty}V(z) = 0, \quad\text{and}\quad
	\lim_{z\nearrow+\infty}V(z) = 1.
\]
This gives a heteroclinic orbit for \eqref{5E5} as claimed. Transforming to $u$ gives the desired solution to \eqref{5E1}.

Finally, the uniqueness (up to translation of $z$) follows from the uniqueness of the orbit near $Y=(0,0)$.

%%%%%%%
%%%%%%%
 
\subsubsection{The case $a =1/2$}

In this case, as $dE/dz = 0$ the energy is constant along an orbit and the equation is reduced to a first order equation
\[
	\frac12\dot{V}^2 + F(V) = E_0
\]
where $E_0$ is a constant. Taking $E_0=E(0,0)$ to get desired homoclinic orbit; see the bold curve in Figure \ref{F1}. In other words, there is a global (positive) solution $V$ of \eqref{5E5} such that
\[
	\lim_{z\to \pm\infty}V(z) = 0.
\]
The uniqueness (up to translation) follows from the uniqueness of the orbit. Since the translation in $z$ corresponds to the scaling of $u$, the desired uniqueness follows.

\begin{remark} \label{5TR}
For $a=1/2$, one can make use of shooting method to obtain the same conclusion. To be more precise, one can prove that there exists a decreasing solution $V$ of \eqref{5E5}, which satisfies $V(0)>1$ $\dot{V}(0)=0$ and $V(z)=V(-z)$ for $z>0$ and $V(z)\searrow 0$ as $|z|\nearrow +\infty$. This implies that there is a positive global solution $u$ to \eqref{5E1} such that
\[
	\lim_{x\searrow0} \frac{u(x)}{u_a(x)} = 0 \quad\text{and}\quad
	\lim_{x\nearrow+\infty} \frac{u(x)}{u_a(x)} = 0.
\]
For interested readers, we refer to \cite{G} for similar arguments. It is worth noting that this type of method is useful if the system is not autonomous. Apparently, in our current autonomous problem, the phase plane approach is easier.
\end{remark}

\begin{remark}\label{5TR-new}
In the above proof, one takes $E_0=E(0,0)$ to get a homoclinic orbit, now if $E_0$ is taken so that $E_0<E(0,0)$, then one gets a periodic solution; see concentric curves near $(1,0)$ in Figure \ref{F1}. This implies the existence of another type of a positive solution to \eqref{5E1} different from that in Theorem \ref{5T2}. For convenience, we state this new existence result as a theorem and put it below this remark.
\end{remark}

\begin{theorem}\label{5T2-new}
Assume that $a=1/2$. Then there is a positive solution u of \eqref{5E1} such that $u-u_a$ changes its sign infinitely many times near $x=0$ and that $u/u_a$ is bounded in $(0,+\infty)$ but $u/u_a$ does not converge as $x$ tends to zero.
\end{theorem}

It is worth noting that the existence of oscillating solutions mentioned in Theorem \ref{5T2-new} can also be proved by using Sturm's comparison principle, see Lemma \ref{5T7}, without using the phase plane analysis. We leave the detail for interested readers.

%%%%%%%%%%%%%%%%%%%%%%%%%%%%%%%%%%%%%%%%%%%%
\subsection{Final remarks} 
\label{5NU7}
% Subsection 1.6

In this final section, we give here several speculations.

%%%%%%%%%%%%
%%%%%%%%%%%%

\subsubsection{An explicit solution in the case \ref{5C1}}

Let us discuss the example given in the introduction; see \eqref{eqExample}. In this case, we have $\sigma=1$ and $p=-4$ so that $a=(2+\sigma)/(1-p)=3/5$ and $2a-1=1/5$. The characteristic roots $\mu_+>0>\mu_-$ are solutions of
\[
	\mu^2 + \frac{1}{5}\mu - \frac{6}{5} = 0
\]
since $c^{p-1}(1-p)=(3/5)(3/5-1)5= -6/5$. Thus 
$$\mu_\pm=\frac{1}{2}\left(-\frac{1}{5}\pm\frac{1}{5}\sqrt{121}\right)=\frac{1}{10}(-1\pm11),$$ so $\mu_+=1$ and $\mu_-=-6/5$. So the explicit one-parameter family of solutions is
\begin{equation}\label{5Ep}
u^\alpha (x) = \big( \frac{25}{6} \big)^{1/5} x^{3/5}(1+\alpha x)^{2/5},
\end{equation}
which clearly satisfies our asymptotic form \eqref{5E13}.

%%%%%%%%%%%%
%%%%%%%%%%%%

\subsubsection{Asymptotic behavior of constructed solutions in the case \ref{5C1}}

In the case \ref{5C1} so that $p<1$, the solution $u^\alpha$ constructed in Theorem \ref{5T1} seems to have linear asymptotic behavior, which is also the maximal growth of solution at infinity; see Lemma \ref{lem-basic-n=1}. Indeed, as $p<1$, for sufficiently large $V$ the equation \eqref{5E5} can be approximated by
\begin{equation*} \label{5E19}
\ddot{V} + (2c-1) \dot{V} - c^{p-1}_a V = 0.
\end{equation*}
The characteristic roots of
\[
\mu^2 + (2c-1)\mu - c^{p-1}_a = 0
\]
equal
\[
\frac{1}{2} \Big( -(2a-1) \pm \sqrt{(2a-1)^2 - 4c^{p-1}_a} \Big).
\]
Since
\[
(2a-1)^2 + 4c^{p-1}_a = (2a-1)^2 + 4a(1-a) = 1,
\]
the two roots are $1-a>0$ and $-a<0$. The asymptotic behavior of $V$ as $z\nearrow +\infty$ equals $\exp\left((1-a)z\right)$. If one converts to $v$, this must be $x^{1-a}$ (up to a multiple constant). Thus 
\[
u^\alpha = vu_a \sim \operatorname{const} \cdot x
\]
as $x\nearrow +\infty$.

%%%%%%%%%%%%
%%%%%%%%%%%%

\subsubsection{Asymptotic behavior of solutions in the case \ref{5C1} near $x=0$}

Recall that the solution constructed in Theorem \ref{5T1} enjoys the asymptotic behavior
\[
\lim_{x\searrow 0}\frac{u(x)}{u_a(x)}=1.
\]
An immediate consequence of this is that $u(0)=0$. In this section, by providing an explicit solution $u$ to \eqref{5E1}, we show that the equation \eqref{5E1} admits solutions which do not enjoy the above asymptotic behavior. Although the argument below works for any case of $\sigma$ and $p$ satisfying \ref{5C1}, for simplicity, we only consider the equation \eqref{5E1} in the case $\sigma = 0$ and $p=-4$, namely $-u''(x) = u(x)^{-4}$. It is easy to verify that 
\[
u(x) = \frac{180^{2/5}}{6} (1+x)^{2/5}
\]
solves \eqref{5E1} in this particular case. The above solution is simply the dual of the solution given in \eqref{5Ep}. Obviously $u(0)>0$. Notice that in this case we have $a=2/5$, leading $u_a (x)= c_{2/5} x^{2/5}$. Hence
\[
\lim_{x \searrow 0} \frac{u(x)}{u_{2/5}(x)}=+\infty.
\]
Thus, the above solution does not belong to the one-parameter family of solutions $\{u^\alpha\}_{\alpha>0}$ constructed in Theorem \ref{5T1}.

%%%%%%%%%%%%
%%%%%%%%%%%%

\subsubsection{Asymptotic behavior of solutions in the case \ref{5C2} near $x=0$}

In contrast to the case \ref{5C1} discussed above, we show in this section that in the case \ref{5C2} any positive solution $u\in C[0,+\infty)\cap C^2(0,+\infty)$ to \eqref{5E1} always satisfies $$\lim_{x\searrow 0}u(x)=0$$ since $\sigma<-2$ and $p>1$.
Indeed, in view of Lemma \ref{lem-basic-n=1} and by way of contradiction, there holds
\begin{equation}\label{5E20}
\lim_{x\searrow 0}u(x)=k>0.
\end{equation}
Then, there is some small $\delta>0$ and small $\epsilon>0$ such that
$$ u(x)^p \geq \epsilon \quad \text{for all } x \in (0, \delta].$$
Integrating the equation \eqref{5E1} gives
\[
u(x) = (x-\delta) u'(\delta) - \int_x^\delta \int_y^\delta s^\sigma u(s)^p ds dy
\]
for $x \in (0, \delta)$. As 
\[
\int_x^\delta \int_y^\delta s^\sigma u(s)^p ds dy
\geq \epsilon \int_x^\delta \int_y^\delta s^\sigma ds dy
=\epsilon \frac{ x^{\sigma+2} - x (\sigma+2)\delta^{\sigma+1} + (\sigma+1)\delta^{\sigma+2} }{(\sigma+1)(\sigma+2)},
\]
we deduce that
\[
u(x)
\leq (x-\delta) u'(\delta) - \epsilon \frac{ x^{\sigma+2} - x (\sigma+2)\delta^{\sigma+1} }{(\sigma+1)(\sigma+2)}
\]
for all $x \in (0, \delta)$. Keep in mind that $\sigma <-2$. Hence by letting $x \searrow 0$ we easily obtain $u(x) \searrow -\infty$, which is clearly a contradiction. Hence, \eqref{5E20} holds. Notice that the one-parameter family of solutions constructed in Theorem \ref{5T1} also enjoys 
$$\lim_{x\searrow 0}u^\alpha (x)=0,$$
because $u^\alpha \sim u_a$ near zero; but, different from the case \ref{5C1}, \eqref{5E20} may not imply that the limit $u(x)/u_a(x)$ exists as $x\searrow 0$. In fact, as stated in Theorem \ref{5T2-new}, there is an oscillatory solution $u$ such that $u(x)/u_a(x)$ is oscillating as $x\searrow 0$ at least for $a=1/2$.

%%%%%%%%%%%%
%%%%%%%%%%%%

\subsubsection{No solution strictly below $u_a$ in the case \ref{5C1}} 
\label{5NU4}

In this section, we show that no solution to \eqref{5E1} stays strictly below $u_a$ in the case \ref{5C1}. The argument is based on a version of Sturm's comparison principle for oscillation, which will be mentioned below, which is similar to that of \cite{G}. 

\begin{lemma}\label{5T7}
Suppose that $u$ and $v$ solve
\begin{align*}
(\sigma u')' + \sigma q_1 u \geq 0 \geq (\sigma v')' + \sigma q_2 v \label{5E16}
\end{align*}
on $(a,b)$, where $\sigma>0$ and $q_1,q_2$ are in $C[a,b]$. Suppose that
\[
u>0 \quad \text{on} \; (a,b) \quad \text{and} \quad u(b)=0.
\]
Assume that
\[
	v(a)u'(a) - u(a)v'(a) \geq 0 \quad \text{and} \quad v(a) \geq 0
\]
and 
\[
\text{either} \quad v(a)>0 \quad \text{or} \quad v'(a)>0.
\]
If $q_1\leq q_2$ on $(a,b),$ then $v \; \text{ has zero in } \; (a,b)$ unless $q_1\equiv q_2$ and $u$ and $v$ are solutions to $(\sigma w')' + \sigma q_1 w =0$ in $(a,b)$.
\end{lemma} 

Compared with the version in \cite{G}, our set of hypotheses in Lemma \ref{5T7} is slightly different. However, one can mimic the argument used in \cite{G} to prove the above lemma without difficulty. For completeness, we sketch its proof. For simplicity, we denote
\[
(\sigma u')' + \sigma q_1 u = f_1 \geq 0 \geq -f_2 = (\sigma v')' + \sigma q_2 v .
\]
If $v$ has no zero on $(a,b)$, then from the hypotheses on $v$ we deduce that $v>0$ on $(a,b)$. A simple calculation shows
\[
v f_1 + uf_2 = v(\sigma u')' - u(\sigma v')' + \sigma (q_1 -q_2) u v.
\]
Integrating by parts the above identity over $(a,b)$ yields
\[
- \int_a^b (v f_1 + uf_2) + \big[\sigma (vu' - u v') \big] \Big|_a^b +\int_a^b \sigma (q_1 -q_2) u v = 0
\]
Since $u'(b) \leq 0$, the boundary conditions yields
\[ 
 \big[ \sigma(vu' - uv') \big] \Big|^{b}_{a} 
=v(b)u'(b) - u(b)v'(b) - \big[ v(a)u'(a) - u(a)v'(a)\big] \leq 0.
\]
This leads a contradiction since the other two terms are strictly negative unless $q_1=q_2$ and $f_1=f_2=0$. This completes the proof.
% 原稿 8/16

As an application of Lemma \ref{5T7}, we obtain a non-existence for a positive solution strictly below $u_a(x)$ under the condition \ref{5C1}. For convenience, we state this as a theorem.

\begin{theorem} \label{5T8}
Assume \ref{5C1}. Then, there is no (global) positive solution to \eqref{5E1} strictly smaller than $u_a$.
 In particular, $u^\alpha$ in \eqref{5E13} for $\alpha<0$ cannot be extended to a global positive solution to \eqref{5E1}.
\end{theorem}
\begin{proof}
The first statement is equivalent to saying that there is no global solution $V$ to \eqref{5E5} satisfying $0<V<1$. By changing the independent variable $z$ to $-z$, if necessary, we may assume that $2a-1\geq0$. Furthermore, note that $V^p-V>0$ for $0< V<1$ as $p<1$.

We first observe that $V$ must be non-increasing. Indeed, it follows from the equation \eqref{5E5} that if $\dot{V}(z_0)=0$ at some $z_0 > 0$ then $\ddot{V}(z_0)<0$. This implies that $V$ has no local minimum. If there is a local maximum of $V$ at $z_0$, namely $\dot{V}(z_0)=0$ and $\ddot{V}(z_0) \leq 0$, then $V$ must be non-decreasing in $(-\infty,z_0)$. Otherwise, there is some local minimum in $(-\infty, z_0)$, which is a contradiction. Hence, as $V^p-V>0$ and $(2a-1)\dot{V}\geq0$, there holds $\ddot{V}<0$ by \eqref{5E5}. This means $V$ is concave on $(-\infty,z_0)$. This is impossible since we assume that $V>0$ on $(-\infty,z_0)$; see also section \ref{ssect-thm-n=1-R}. Thus, we have just shown that $V$ is either non-decreasing or non-increasing in $\R$. In the case $V$ is non-decreasing, namely $\dot{V} \geq 0$, as before $V$ must be concave by \eqref{5E5}, which contradicts $V>0$. Thus, $V$ must be non-increasing.

We next observe that $\lim_{z \nearrow +\infty}V(z)=0$. Indeed, it is now clear that the limit $\lim_{z \nearrow +\infty}V(z)=:k$ exists. By way of contradiction, suppose $k>0$. Then by the monotonicity of $V$, we further have $k < 1$. Integrating \eqref{5E5} from $(1,z)$ yields
\begin{equation}\label{5E-integral}
	\dot{V}(z) - \dot{V}(1) + (2a-1) \left(V(z)- V(1)\right)
	+ c^{p-1}_a \int^z_1 \left(V^p(s)- V(s)\right) ds = 0.
\end{equation}
As $k \in (0,1)$, the integral term in \eqref{5E-integral} diverges to $+\infty$. In addition, as $V$ is non-increasing and bounded as $z\nearrow +\infty$, there is a sequence $z_j$ such that $\dot{V}(z_j)\to0$. This and the above identity lead a contradiction. 

With $\sigma(z)=\exp(2a-1)z$ and $q_2=c^{p-1}_a(V^{p-1}-1)$, we rewrite the equation \eqref{5E5} as
\[
	\frac{d}{dz} (\sigma\dot{V}) + \sigma q_2 V = 0.
\]
Take $m>(2a-1)^2/4$ so that $\mu^2+(2a-1)\mu+c^{p-1}_a m=0$ has a non-real zero.
 We compare with
\[
	\frac{d}{dz} (\sigma\dot{V}_0) + \sigma m V_0 = 0.
\]
This is a linear equation and its solution is of the form
\[
	V_0(z) = B \exp \Big(- \frac{(2a-1)z}2\Big) \cos \left(K(z-z_*)\right)
\]
with $K=(1/2) \sqrt{4m-(2a-1)^2}$ for some $B, z_*\in\R$. We take $z_2$ large enough so that $q_2(z_2)>m$.
This is possible since $V(z)\to0$ as $z\to0$ and $p<1$ so that $V^{p-1}\nearrow +\infty$ as $v\to0$.
Since $\dot{V}(z)\leq0$, we conclude that $q_2(z)>m$ for $z\in[z_2\ \infty)$.
We take $V_0$ so that 
$$V_0(z_2)=V(z_2), \quad V'_0(z_2)=V'(z_2).$$
In view of the form of $V_0$, we let $b$ be the first zero of $V_0$ for $b>z_2$.
 We apply Lemma \ref{5T7} with 
\[
q_1=m, \quad a=z_2, \quad u=V_0, \quad v=V, \quad f_1 \equiv f_2 \equiv 0
\]
 to conclude that $V$ must have zero in $(z_2,b)$.
 This contradicts the assumption that $V(z)>0$ for all $z\in\R$.
 We thus conclude that there is no solution to \eqref{5E5} satisfying $0<V<1$.
 This means that there is no global positive solution $u$ to \eqref{5E1} satisfying $u(x)<u_a(x)$ for $x>0$.

We next consider $u^\alpha$ for $\alpha<0$. Near $x=0$, we conclude from Corollary \ref{5T6} and $\alpha<0$ that $w<0$ and $w'<0$. This implies that $u^\alpha$ decreases and is below $u_a$ near $x=0$. Since $u_a$ is increasing and $u^\alpha$ is concave, the extended $u^\alpha$ must satisfy $u^\alpha(x)<u_a(x)$ everywhere. In other words, $u^\alpha$ is below $u_a$ everywhere. This violates what we have established earlier.
\end{proof}

\begin{remark}
In the case $p \leq 0$, the integral term in \eqref{5E-integral} still diverges to $+\infty$. Hence, by the argument leading to $\lim_{z \nearrow +\infty}V(z)=0$ we quickly conclude that there is no global solution $V$ to \eqref{5E5} satisfying $0<V<1$. This avoids using Sturm's comparison principle.
\end{remark}

%%%%%%%%%%%%
%%%%%%%%%%%%

\subsubsection{A phase plane for the transformed equation (\ref{5E5})}
\label{5NU6}

In this last part of the paper, we would like to mention the phase plane analysis for $Y=(V,\dot{V})$. Depending on the sign of $(2a-1)^2 - 4a(1-a)(p-1)$, we have two cases.

\noindent\textbf{Case 1}. Suppose $(2a-1)^2 - 4a(1-a)(p-1) < 0$, namely $p>(4a(1-a) )^{-1}$. In this case and depending on the size of $a$, the phase portrait is as in Figures \ref{F1}--\ref{F3}. We notice that the solutions found in Theorem \ref{5T2} correspond the bold curves in these figures. Moreover, any solution $u$ enjoys
\[
\limsup_{x\nearrow+\infty} \frac{u(x)}{u_a(x)} < +\infty,
\]
which basically says that $u/u_a$ is bounded, but not necessarily by $1$. Note that the solution mentioned in Remark \ref{5TR} also corresponds to the bold curve in Figure \ref{F1}. Still in the case $a=1/2$, as indicate in Theorem \ref{5T2-new} above, there exists a solution $u$ such that $u-u_a$ oscillates around zero. This oscillating solution corresponds to concentric curves near the point $(1,0)$ in Figure \ref{F1}.
\noindent
\begin{figure}[H]
\captionsetup{width=.3\linewidth}
 \begin{minipage}[b]{0.32\linewidth}
 \centering
 \includegraphics[width=4cm]{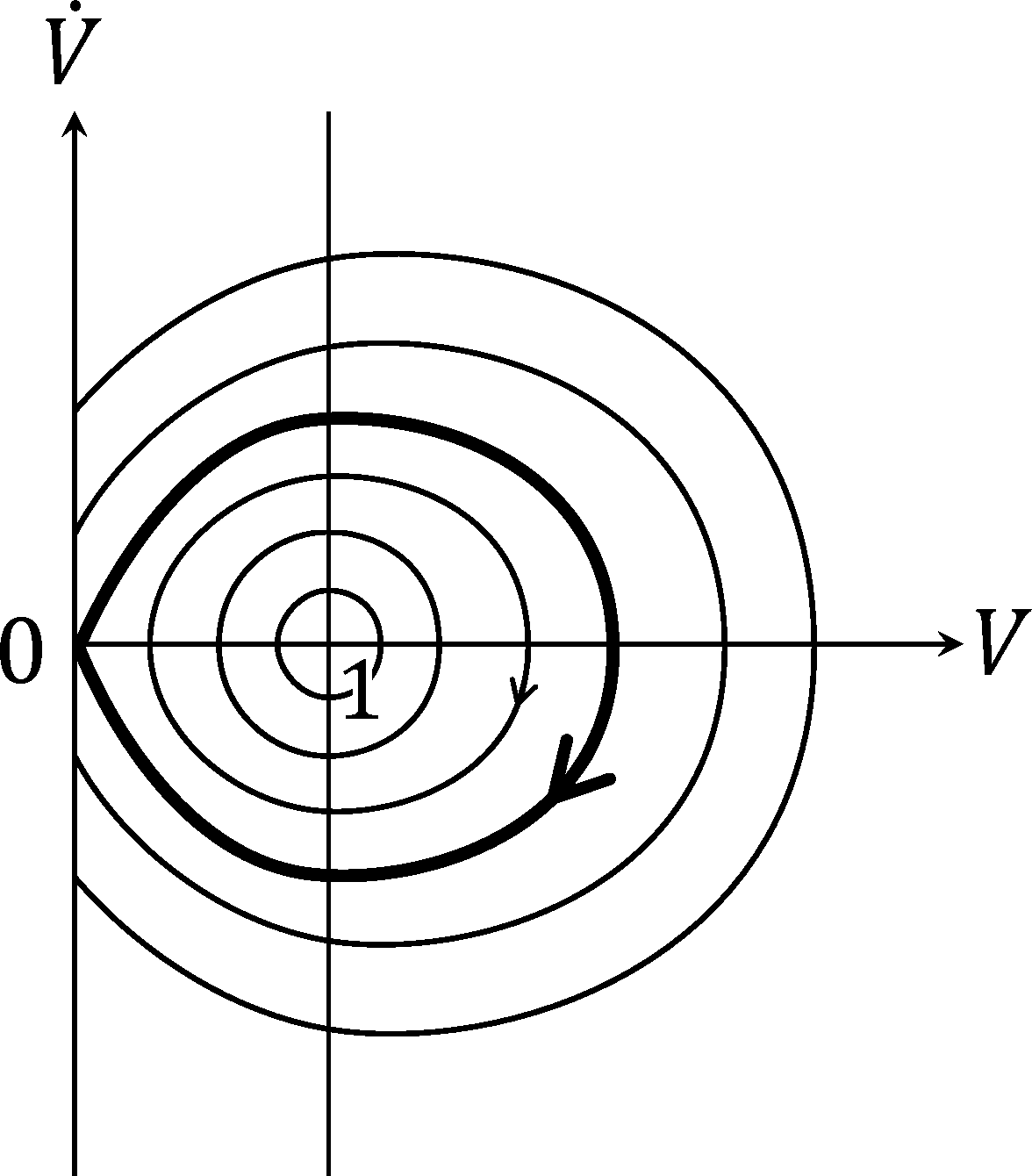}
 \caption{$a=1/2$} \label{F1}
 \end{minipage}
 \begin{minipage}[b]{0.32\linewidth}
 \centering
 \includegraphics[width=3.75cm]{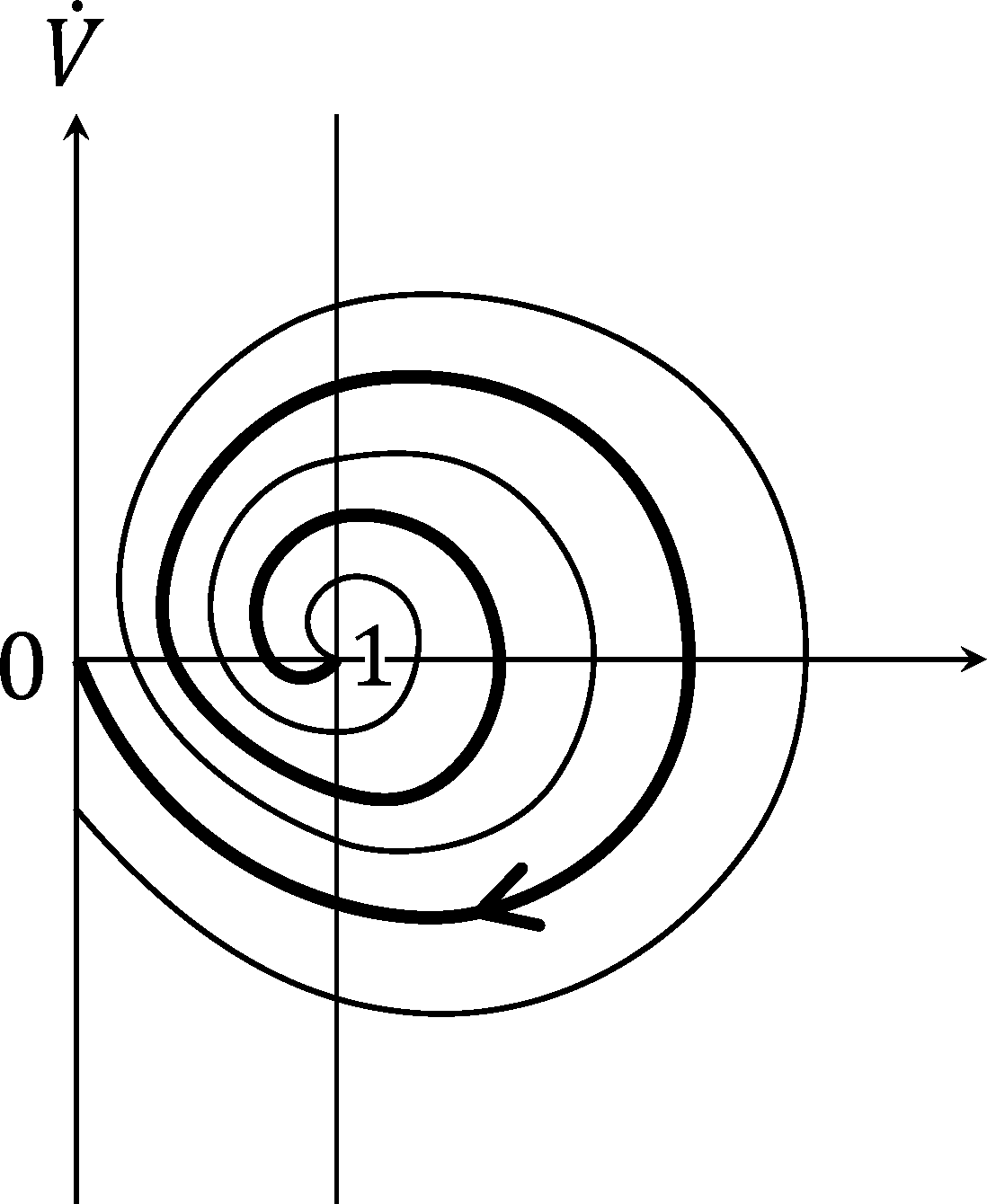}
 \caption{$a<1/2$} \label{F2}
 \end{minipage}
 \begin{minipage}[b]{0.32\linewidth}
 \centering
 \includegraphics[width=4cm]{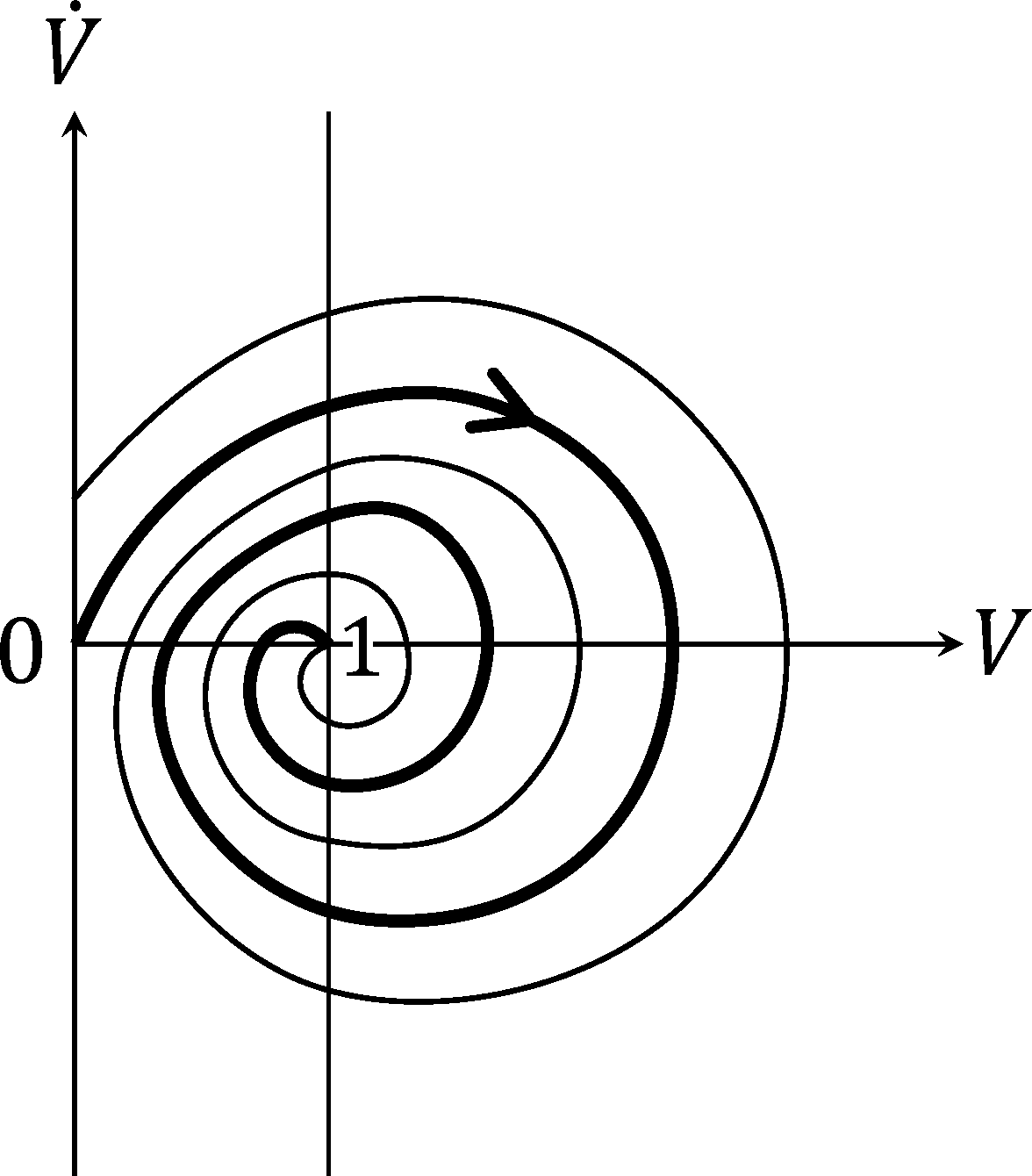}
 \caption{$a>1/2$} \label{F3}
 \end{minipage}
\end{figure}

\noindent\textbf{Case 2}. Suppose $(2a-1)^2 - 4a(1-a)(p-1) \geq 0$, namely $p \leq(4a(1-a) )^{-1}$. In this case, the phase portrait is as in Figures \ref{F4} and \ref{F5}. We notice that the solutions constructed in Theorem \ref{5T2} in the case $a>1/2$ corresponds to the bold curve in Figure \ref{F5}. In this particular case, there holds $u/u_a \leq 1$.
\noindent
\begin{figure}[H]
 \begin{minipage}[b]{0.49\linewidth}
 \centering
 \includegraphics[width=4.25cm]{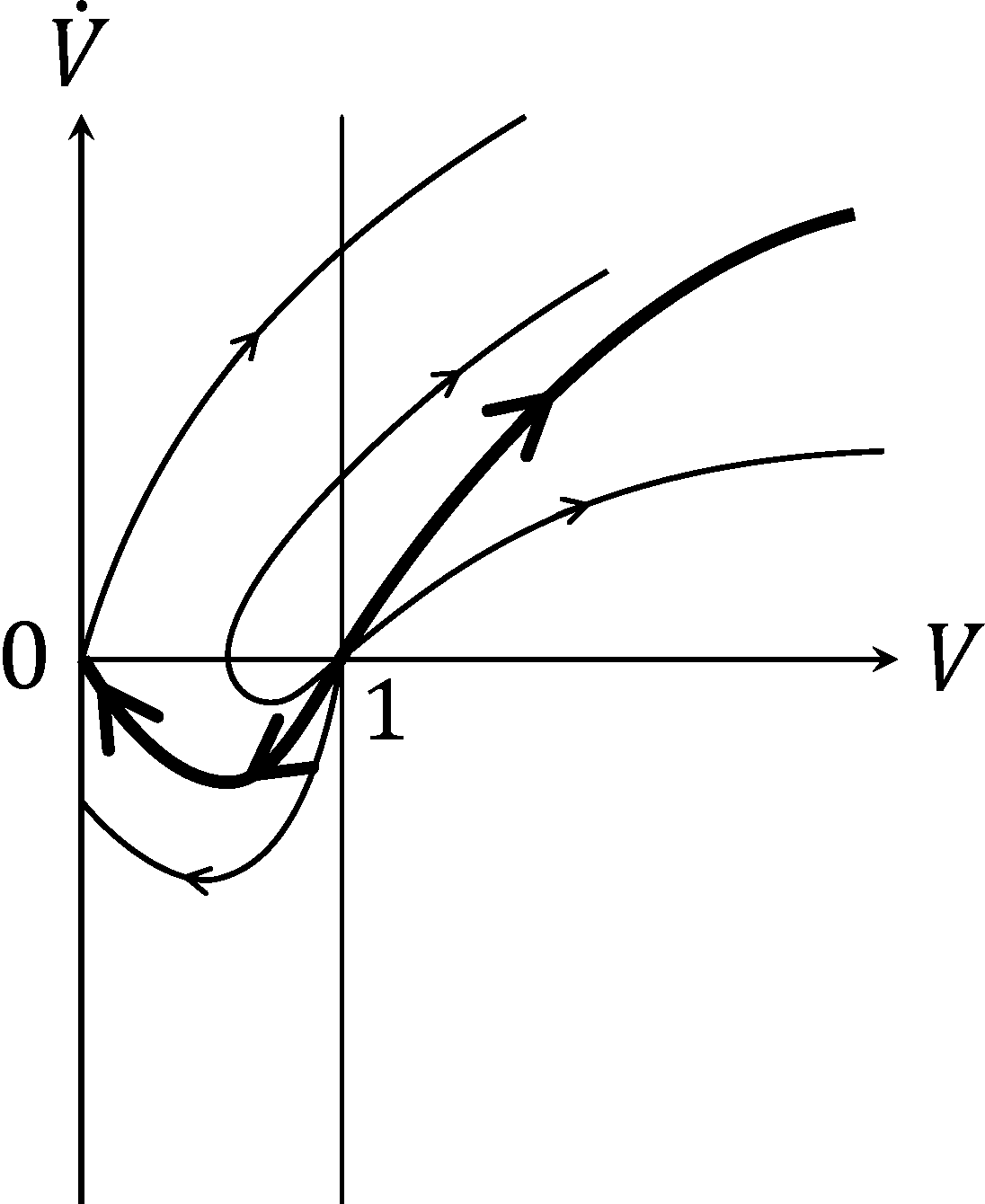}
 \caption{$a<1/2$} \label{F4}
 \end{minipage}
 \begin{minipage}[b]{0.49\linewidth}
 \centering
 \includegraphics[width=4.25cm]{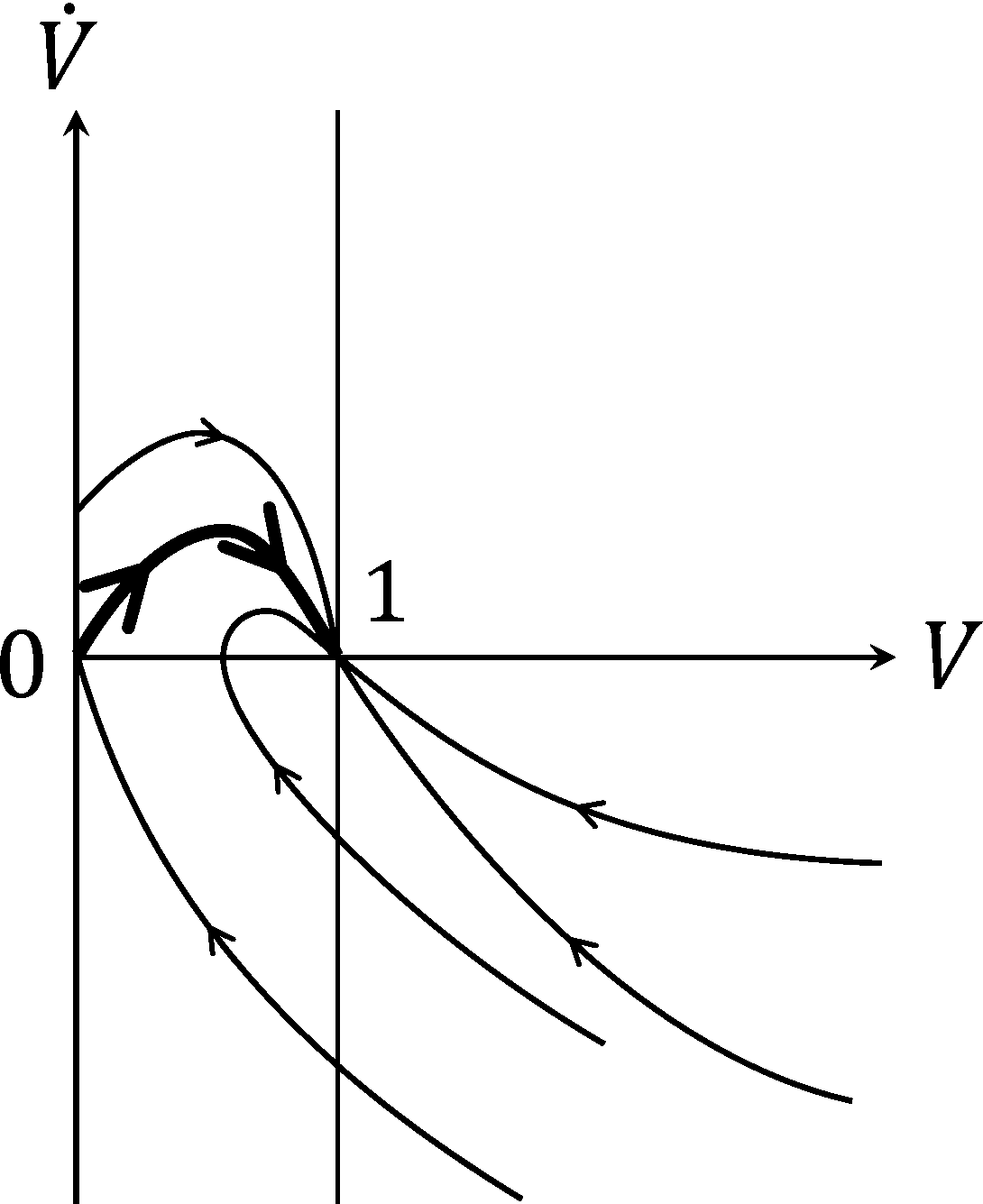}
 \caption{$a>1/2$} \label{F5}
 \end{minipage}
\end{figure}
It is worth noting that the bold curve on the right side of the vertical line $V=1$ in Figure \ref{F4} indicates that for $a<1/2$ there is a positive solution $u$ of \eqref{5E1} such that $u\geq u_a$. Again, keep in mind that the case $a=1/2$ is excluded since $p$ must be $p\leq1$. 

%%%%%%%%%
%%%%%%%%%

\section*{Acknowledgments}

The first author is grateful to Professor Eiji Yanagida for letting him know \cite{Pe} about the Lienard system. The work of the first author was partly supported by the Japan Society for the Promotion of Science through the grants KAKENHI No. 19H00639, No. 18H05323, No. 17H01091, and by Arithmer Inc. and Daikin Industries, Ltd. through collaborative grants. The work of the second author was partly supported by the Mathematical Society of Japan.

\section*{ORCID IDs}

\noindent Yoshikazu Giga: 0000-0003-3048-5240\\
\noindent Qu\cfac oc Anh Ng\^o: 0000-0002-3550-9689

\end{document}